\documentclass[12pt,oneside]{amsproc}

\usepackage[utf8]{inputenc}
\usepackage[english]{babel}
\usepackage{comment}
\usepackage[alphabetic]{amsrefs}
\usepackage{tikz}
\usepackage{xcolor}
\usepackage[a4paper,margin=30mm]{geometry}
\usepackage{stackengine}
\usepackage{scalerel}
\usepackage{mathtools}
\usepackage{calc}
\usepackage{amsthm}
\usepackage{thmtools}
\usepackage{enumitem}
\usepackage{todonotes}

\usepackage{amssymb}
\usepackage{amsfonts}

\usepackage{euscript}
\usepackage{stix}
\usepackage{dsfont}  

\theoremstyle{plain}
\newtheorem{theorem}{Theorem}
\newtheorem{lemma}[theorem]{Lemma}
\newtheorem{proposition}[theorem]{Proposition}
\newtheorem{corollary}[theorem]{Corollary}
\newtheorem{claim}{Claim}[theorem]

\theoremstyle{definition}
\newtheorem{definition}[theorem]{Definition}

\theoremstyle{remark}
\newtheorem{remark}[theorem]{Remark}
\newtheorem{example}[theorem]{Example}

\declaretheoremstyle[
  spaceabove=6pt, 
  spacebelow=6pt, 
  headfont=\normalfont\itshape, 
  bodyfont = \normalfont,
  postheadspace=1em, 
  qed=\qedsymbol, 
  headpunct={.}]
{myproof} 
\declaretheorem[style=myproof, unnumbered]{proof}

\newenvironment{enumerate-(a)}{\begin{enumerate}[label={\upshape (\alph*)}, leftmargin=2pc]}{\end{enumerate}}

\newenvironment{enumerate-(a)-r}{\begin{enumerate}[label={\upshape (\alph*)}, leftmargin=2pc,resume]}{\end{enumerate}}

\newenvironment{enumerate-(A)}{\begin{enumerate}[label={\upshape (\Alph*)}, leftmargin=2pc]}{\end{enumerate}}

\newenvironment{enumerate-(A)-r}{\begin{enumerate}[label={\upshape (\Alph*)}, leftmargin=2pc,resume]}{\end{enumerate}}

\newenvironment{enumerate-(i)}{\begin{enumerate}[label={\upshape (\roman*)}, leftmargin=2pc]}{\end{enumerate}}

\newenvironment{enumerate-(i)-r}{\begin{enumerate}[label={\upshape (\roman*)}, leftmargin=2pc,resume]}{\end{enumerate}}

\newenvironment{enumerate-(I)}{\begin{enumerate}[label={\upshape (\Roman*)}, leftmargin=2pc]}{\end{enumerate}}

\newenvironment{enumerate-(I)-r}{\begin{enumerate}[label={\upshape (\Roman*)}, leftmargin=2pc,resume]}{\end{enumerate}}

\newenvironment{enumerate-(1)}{\begin{enumerate}[label={\upshape (\arabic*)}, leftmargin=2pc]}{\end{enumerate}}

\newenvironment{enumerate-(1)-r}{\begin{enumerate}[label={\upshape (\arabic*)}, leftmargin=2pc,resume]}{\end{enumerate}}

\newenvironment{itemizenew}{\begin{itemize}[leftmargin=2pc]}{\end{itemize}}

\newcommand{\forallH}{\forall\raisebox{1.1ex}{\scaleto{\sf H}{.8ex}\kern-.2ex}}
\newcommand{\existsH}{\exists\raisebox{1.1ex}{\scaleto{\sf H}{.8ex}\kern-.2ex}}
\newcommand{\forallS}{\forall\raisebox{1.1ex}{\scaleto{\sf S}{.8ex}\kern-.2ex}}
\newcommand{\existsS}{\exists\raisebox{1.1ex}{\scaleto{\sf S}{.8ex}\kern-.2ex}}
\newcommand{\LL}{\EuScript L}
\newcommand{\HH}{\EuScript H}
\newcommand{\FF}{\EuScript F}
\newcommand{\MID}{\mathbin{:}}

\author{Claudio Agostini}
\author{Stefano Baratella}
\author{Silvia Barbina}
\author{Luca Motto Ros}
\author{Domenico Zambella}

\address[Claudio Agostini]{Institut f\"ur Diskrete Mathematik und Geometrie, TU Wien}
\address[Stefano Baratella]{Dipartimento di Matematica, Universit\`{a} di Trento}
\address[Silvia Barbina]{Sezione di Matematica, Universit\`{a} di Camerino}
\address[Luca Motto Ros, Domenico Zambella]{Dipartimento di Matematica, Universit\`{a} di Torino}

\title{Continuous logic in a classical setting}

\thanks{Claudio Agostini was partially supported by the FWF grant P 35655. Silvia Barbina and Luca Motto Ros were partially supported by PRIN2022 \textit{Models, sets and classifications}, prot. 2022TECZJA. Domenico Zambella was partially supported by PRIN2022 \textit{Logical methods in combinatorics}, prot. 2022BXH4R5}

\begin{document}

\begin{abstract}
Let ${\LL}$ be a first-order two-sorted language and consider a class of $\LL$-structures of the form $\langle M,  X \rangle$ where $M$ varies among structures of the first sort,  while $X$ is  fixed in the second sort, and it is assumed to be a compact Hausdorff space.  When $X$ is  a compact subset of the real line, one way to treat classes of this kind model-theoretically is via continuous-valued logic, as in \cite{BYBHU}.
Prior to that, Henson and Iovino proposed an approach based on the notion of positive formulas~\cite{HI}.
Their work is tailored to the model theory of Banach spaces.
Here we show that a similar approach is possible for a more general class of models.
We introduce suitable versions of elementarity, compactness, saturation, quantifier elimination and other basic tools, and we develop basic model theory.   
\end{abstract}

\maketitle
\raggedbottom



\section{Introduction}\label{intro}

\def\ceq#1#2#3{\parbox[t]{23ex}{$\displaystyle #1$}\parbox{6ex}{\hfil $#2$}{$\displaystyle #3$}}

As a motivating example for our work, consider a first-order structure $M$ equipped with a bounded distance $d: M^2\rightarrow [0,1]$, where $[0,1]$ is the real unit interval.   We  regard $d$  as a function in a two-sorted structure $\langle M, [0,1]\rangle$. We introduce a suitable two-sorted language $\LL$  and we require that the formulas in this language should remain  classically  (i.e. two-) valued, as opposed to what happens in continuous logic, where the set of truth values is the real unit interval $[0,1]$ --- see~\cite{BYBHU}  for an introduction to continuous logic and~\cite{K} for a generalisation  to structures without a metric.

More generally, we  deal with structures of the form $\langle M,  X \rangle$, where $M$ is some first-order structure and $X$ is a fixed compact Hausdorff space, also treated as a first-order structure, in a way that we detail in Section~\ref{uno}. 
In this regard, we point out  again that --- differently from \cite{CK}, where  formulas take values in an arbitrary compact Hausdorff space --- in our setting formulas are classically valued.

Our goal is to be able to apply basic model-theoretic results, such as compactness and the existence of elementary and saturated extensions, to structures like  $\langle M,  X\rangle$, while requiring that $X$ should stay fixed. Ours is an attempt to reconcile these two conflicting requirements to a certain extent.   For example,  with reference to our motivating example, when considering a saturated extension of $\langle M, [0,1] \rangle$, we would like the metric to still take values in $[0,1]$, rather than in an elementary extension that would possibly contain nonstandard elements. 

We treat the class of structures described above with an approach similar to~\cite{HI}. As in \cite{HI}, we isolate a suitable subclass of formulas $\FF\subseteq\LL$ with respect to which we prove a compactness theorem and we formulate appropriate versions of notions such as elementary extension and saturation.
We point out a difference from \cite{HI}.
In the class of formulas $\FF$ we explicitly allow quantification over elements of the topological space $X$.
Though we show that these quantifiers can be eliminated, their presence simplifies the comparison with real-valued logic--see Example~\ref{ex_Rvlogic}.
A second difference is discussed after Remark~\ref{rem3}.

We also borrow tools from nonstandard analysis, more precisely from nonstandard topology, to define what we call the \emph{standard part} of a structure.  The latter plays a crucial role in the proof of our compactness theorem. See \cite{D} for a concrete introduction to nonstandard analysis.

We give a brief outline of the paper. In Section~\ref{uno}, we describe our two-sorted language $\LL$, and the subclass of $\LL$-structures in which we are interested, namely those of the form $\langle M,  X\rangle$, where $M$ is any structure and $X$ is a given nonempty compact Hausdorff space. 

The language $\LL$ contains, but is not limited to, symbols for continuous function from $ X^n$ (equipped with the product topology) to $ X$ for $n \ge 0$ -- the case $n=0$ gives a constant for each element of $X$ -- and predicate symbols for the closed (equivalently: compact) subsets of $ X^n$ , for all $0<n$.

We also describe the subset $\FF$ of $\LL$-formulas with respect to which we later define our versions of elementarity and saturation: the set $\FF$ restricts  negations in such a way that the structure $ X$ stays the same in elementary extensions of $\langle M, X\rangle$. Section~\ref{ultrapws} introduces the notions of an \emph{approximation} of a formula in $\FF$,  and of \emph{approximate satisfiability}, used in Section~\ref{morphisms} to define $\FF$-elementary maps. 
In order to partly recover the connective negation, in  Section~\ref{ultrapws} we also introduce the notion of \emph{strong negation} of an $\FF$-formula.

In Section~\ref{standard_part}, we use tools from nonstandard analysis to define the \emph{standard part} $\langle N,  X \rangle$ of an elementary extension $\langle N, \,^* X\rangle$ of a structure $\langle M,  X \rangle$, thus providing the basis for the compactness theorem proved in Section~\ref{compactness} and the existence of saturated structures. In Section~\ref{compactness} we also  prove that approximate satisfiability is the same as satisfiability in models that are $\omega$-saturated with respect to the class of $\FF$-formulas.

These results allow us, in Section~\ref{monster}, to work in a monster model and to formulate the  analogue in our context of classical quantifier elimination. More precisely, we show that  each  $\FF$-formula formula $\varphi$ and each approximation $\varphi'$ of  $\varphi$ admit an interpolant $\FF$-formula that does not contain quantifiers of the sort which is associated to the elements of $X$.

In Section~\ref{monster} we formulate analogues of the Tarski-Vaught Test and the L\"owenheim-Skolem Theorem. We also introduce some technical notions that allow us to obtain, among  other results, a form of the Perturbation Lemma proved in~\cite{HI} in the monster model.
We conclude this section by applying these notions to bounded metric spaces.

\section{A class of two-sorted structures}\label{uno}

We adopt standard model-theoretic notation and conventions.
In particular,  when there is no ambiguity, we identify a language with the set of its formulas and a structure with its domain.

\medskip

 We fix a nonempty, compact Hausdorff space $X$. Recall that $X$ is a normal space.  We always assume that $ X^n$ is endowed with the product topology, for every \( n > 1 \). 
 
 We associate to $X$ a first-order language $\LL_{\sf S}$.  The subscript  $\sf S$ denotes the sort of this language, which we call the \emph{space sort}.
 The non-logical symbols of $\LL_{\sf S}$ are: 

\begin{itemizenew}
\item  
a set of $n$-ary predicate symbols $C$, one for each $n \geq 1$ and each  closed  subset $C$ of $ X^n$: for each such symbol, we write \( (x_1, \dotsc, x_n) \in C \) instead of \( C(x_1, \dotsc, x_n) \);
\item
a set of function symbols $f$, one for every \( n \geq 0 \) and every (uniformly) continuous $f \colon X^n \to X$.
\end{itemizenew}



We  also use $X$  to denote  the structure with domain $X$ which interprets $\LL_{\sf S}$-symbols as expected;
depending on the context, $C, f$ and $r$ denote either  $\LL_{\sf S}$-symbols or their interpretations in the structure $X$.
An example of  a structure $X$ as above  is the real unit interval $[0,1]$, equipped with all closed sets \( C \subseteq [0,1]^n \), all continuous functions \( f \colon [0,1]^n \to [0,1] \), and all constants \( r \in[0,1] \). We always assume that  $[0,1]$ is equipped with its standard topology.

Next we fix an arbitrary first-order language $\LL_{\sf H}$, which we call the language of the \emph{home sort} $\sf H$.
Finally, we let ${\LL}$ be a full two-sorted language that expands both $\LL_{\sf H}$ and $\LL_{\sf S}$, where the only new symbols  are  function symbols of sort ${\sf H}^n\to {\sf S}$, for some \( n \geq 1 \).

\begin{definition}\label{def_0}
A \emph{standard structure} is a two-sorted ${\LL}$-structure of the form $\langle  M, X\rangle$, where $M$ is some $\LL_{\sf H}$-structure and  $X$ is the structure fixed above.
\end{definition}
  
Since \( X \) is fixed, we often write \( M \) instead of \( \langle M,X \rangle \).
From now on, the capital letters $M$ and $N$ always denote standard structures.
Unless otherwise specified, we use the letters \( a \), \( b \), \dots to denote elements of \( M \), and the Greek letters \( \alpha \), \( \beta \), \dots to denote elements of \( X \).
  
The notion of standard structure is inspired by that of  normed space structure introduced in~\cite{HI}. 
In what follows, when comparing our setting with that of~\cite{HI} we  always refer to the case of two-sorted structures therein, with our  space sort $\sf S$ playing the role of the 
real unit interval $[0,1]$. 
However, the component  of sort {\sf H}  of  a standard structure is an arbitrary first-order structure.



 Since we are working in a \emph{classical} model-theoretic setting,
saturated ${\LL}$-structures exist but,  when $X$ is infinite, they are not standard structures in the sense of Definition~\ref{def_0}.
For this reason, we introduce a set ${\FF}$ of formulas such that  $\LL_{\sf H}\subseteq{\FF}\subseteq {\LL}$ and every  standard structure has an ${\FF}$-elementary, ${\FF}$-saturated extension which is itself a standard structure.


\begin{definition}\label{def_LL} 
The set ${\FF}$ is the least subset  of  $\LL$ such that:
\begin{enumerate-(i)}
\item \label{def_LL-i}
 all formulas of $\LL_{\sf H}$ belong to $\FF$;
\item\label{def_LL-ii} 
all formulas of the form $t(x\,;\eta)\in C$ belong to $\FF$, where $C$ is a closed subset of $X^{n}$, for some $n \geq 1$, and $t(x\,;\eta)$ is an $n$-tuple of terms, each  of sort ${\sf H}^{|x|}\times {\sf S}^{|\eta|}\to {\sf S}$;
 \item
 $\FF$ is closed  under the Boolean connectives $\wedge$ and $\vee$;
\item
$\FF$ is closed  under the quantifiers $\forallH$, $\existsH$ of sort ${\sf H}$  and the quantifiers $\forallS$, $\existsS$ of sort ${\sf S}$.
\end{enumerate-(i)}

We denote by ${\HH}$ the subset of  ${\FF}$ formed by the formulas that do not contain quantifiers of sort ${\sf S}$.
\end{definition}

We adopt the convention that, in a term or in a formula,  Latin letters $x,y,z,\dots$ denote tuples of variables   of sort ${\sf H}$ and Greek letters $\eta,\varepsilon,\dots$ denote tuples of variables of sort ${\sf S}$. Hence we can safely drop the superscript from a quantifier.
Formulas of type~\ref{def_LL-i} and~\ref{def_LL-ii} in Definition~\ref{def_LL} are called  $\FF$-atomic. 

\begin{remark}\label{rem3}
In general, the negation of an \( \FF \)-formula \( \varphi \) is not logically equivalent to an \( \FF \)-formula, unless all its \( \FF \)-atomic subformulas of the form \( t \in C \) are such that \( C \) is clopen.
Nevertheless,
the negation of an $\FF$-formula  is a well-formed  $\LL$-formula and can be evaluated in any (standard) structure.


Formally, the set \( \FF \) does not contain equalities of sort ${\sf S}$. However, if $\alpha$ and $\beta$ are in $X$, the equality $\alpha=\beta$ can be expressed as $(\alpha, \beta) \in \Delta$, where $\Delta$ is the diagonal in $X\times X$.

To simplify the exposition we assume that $\LL_{\sf H}$ contains a relation symbol $r_\varphi(x)$ for every formula $\varphi(x)$, and that all $\LL_{\sf H}$-structures below model $r_\varphi(x)\leftrightarrow\varphi(x)$.
 In other words, we assume the Morleyzation of $\LL_{\sf H}$. The notion of $\FF$-atomic formula in Definition~\ref{def_LL} is consistent with this assumption, and we implicitly assume that all $\LL_{\sf H}$-formulas are atomic. Moreover, we identify $\neg r_\varphi$ with $ r_{\neg \varphi}$ etc.  \end{remark}



The ${\HH}$-formulas extend  the so-called positive bounded formulas introduced in~\cite{HI}. Indeed,  the former are more expressive than the latter and  contain all the  $\LL_{\sf H}$-formulas.
This is a double-edged sword: the expressiveness of $\LL_{\sf H}$, even when $\LL_{\sf H}$ is empty (i.e., it only includes equality), simplifies the proof of compactness and ensures many properties of standard structures. On the other hand, it excludes interesting examples such as the classical function spaces. This issue is discussed in \cite{Z}.

In Section~\ref{compactness} we will prove the Compactness Theorem for the larger class of all ${\FF}$-formulas. 
The greater expressive power of $\FF$ is useful. 
For instance, it is easy to show that the $\sup$-quantifier of real-valued logic can be recovered by means of an $\FF$-formula. 
Somewhat surprisingly, we will also prove (Propositions~\ref{prop_LHapprox1} and~\ref{prop_LHapprox2}) that ${\FF}$-formulas can be approximated, in a sense to be made precise, by  ${\HH}$-formulas.

\begin{example}\label{ex_Rvlogic}
  Let $ X=[0,1]$ be the real unit interval. 
 Let $\dotminus$ and \( \dotplus \) denote, respectively, the cut difference and the cut sum on $ X$.
  Let $t(x)$ be a term of sort ${\sf H}\to {\sf S}$. 
  Then there is an ${\FF}$-formula expressing the equality $\sup_{x} t(x)=\tau$.
  For instance, consider the formula
$$\forall x\ \big(t(x)\dotminus\tau\in\{0\}\big)\,
  {\wedge}\, \forall \varepsilon \Big(\varepsilon\in\{0\}\ \vee\ \exists x\ \big(\tau\dotminus (t(x)\dotplus\varepsilon)\in\{0\}\big)\Big).$$

\end{example}

\section{Approximations and strong negations}\label{ultrapws}

In this section we generalize the notion of \emph{approximation of a positive bounded formula} introduced in~\cite{HI}. We also define the concept of \emph{strong negation} of a formula. These two notions are crucial throughout the paper. 

If $\varphi, \psi$ are $\LL$-formulas, we write $\varphi\to\psi$ to say that the implication is valid in the class of all standard
structures. We also write $\varphi\to\psi\to\eta$ as an abbreviation for $(\varphi\to\psi)\land(\psi\to\eta).$

\medskip

Let $A \subseteq M$ be a subset of an $\LL_{\sf H}$-structure. 
We let ${\LL}(A)$ be the set of formulas in the expansion of  $\LL$ containing a name for each element of $A$. The set of formulas \( \FF(A) \) and \( \HH(A) \) are defined accordingly, following Definition~\ref{def_LL}. 

\begin{definition} \label{def:approximation}
Let \( M \) be a standard structure, and 
let $\varphi,\varphi'$ be ${\FF}(M)$-formulas. We say that $\varphi'$ is an \emph{approximation} of $\varphi$, and we write $\varphi'>\varphi$, if $\varphi'$ is obtained by replacing each \( \FF \)-atomic formula of the form $t\in C$ occurring in $\varphi$  by $t\in C'$, for some closed neighborhood  $C'$ of $C$.
When no such atomic formula occurs in $\varphi$, we stipulate that $\varphi>\varphi$.
\end{definition}

Clearly, $\varphi>\varphi$ also when all \( \FF \)-atomic formulas of the form \( t \in C \) occurring in $\varphi$ are such that $C$ is clopen.
 Moreover, an easy induction on the complexity of \( \varphi \) shows that $\varphi\to\varphi'$ for each approximation  $\varphi'$  of  $\varphi$.

Note also that $>$ is a dense relation on ${\FF}(M)$, in the sense that if $\psi_1>\psi_2$ then there exists $\varphi$ such that $\psi_1>\varphi>\psi_2$. 
Indeed, if $C'$ is a closed neighborhood of a  compact set $C$, then the interior of \( C' \) is an open neighborhood $W$ of $C$ such that $W\subseteq C'$. By 
normality of \( X \), there  is a closed neighborhood $C''$ of $C$ such that $C\subseteq C''\subseteq W\subseteq C'$, so that $C'$ is a closed neighborhood of $C''$. The density of $>$ follows.


\begin{definition} \label{def:strongnegation}
Let \( M \) be a standard structure, and 
let $\varphi$ be an ${\FF}(M)$-formula.
A \emph{strong negation} of \( \varphi \) is an \( \FF(M) \)-formula \( \tilde{\varphi} \) which is obtained from \( \varphi \) in the following way:
\begin{itemizenew}
\item
each \( \FF \)-atomic formula of the form $t\in C$ occurring in $\varphi$ is replaced by an \( \FF \)-atomic formula of the form $t\in\tilde{C}$, where $\tilde{C}$ is some closed set disjoint from $C$;
\item
each \( \FF \)-atomic formula in \( \LL_{\sf H} \) occurring in \( \varphi \) is replaced by 
its negation (see Remark~\ref{rem3});
\item 
each connective, each quantifier of sort \( \sf S \), and each quantifier of  sort \( \sf H \) are replaced by their duals, that is, \( \vee \) and \( \wedge \) are replaced by \( \wedge \) and \( \vee \), respectively, while existential (respectively, universal) quantifiers are replaced by universal (respectively, existential) quantifiers.
\end{itemizenew}
\end{definition}

We write \( \tilde{\varphi} \perp \varphi \) to say that \( \tilde{\varphi} \) is a strong negation of \( \varphi \).
By induction one can check that \( \tilde{\varphi} \to \neg \varphi \), which justifies our terminology. 

\begin{lemma}\label{lem_interpolation}
  For all  $\varphi\in{\FF}(M)$
  \begin{enumerate-(1)}
    \item \label{lem_interpolation-1}
 for every \( \varphi' > \varphi \), there are strong negations \( \tilde{\varphi}_0 < \tilde{\varphi}_1 \) of \( \varphi \) such that
 \[ 
 \varphi \to \neg \tilde{\varphi}_1 \to \neg \tilde{\varphi}_0 \to \varphi' ;
 \]
    \item \label{lem_interpolation-2}
    for every $\tilde{\varphi}\perp\varphi$, there is a formula $\varphi'>\varphi$ such that  $\varphi\rightarrow\varphi'\rightarrow\neg \tilde{\varphi}$.
  \end{enumerate-(1)}
\end{lemma}

\begin{proof}
We prove~\ref{lem_interpolation-1} and~\ref{lem_interpolation-2} by simultaneous  induction on formulas. The case $\varphi\in \LL_{\sf H}$ is straightforward, so we consider the case where $\varphi$ is of the form $t\in C$.
  
Let $\varphi' > \varphi$ be $t\in C'$ for some closed neighborhood \( C' \) of $C$.
 Let $U_0$ be an open set such that $C\subseteq U_0 \subseteq C'$.
By normality of \( X \), pick a closed neighborhood \( C'' \) of \( C \) satisfying \( C'' \subseteq U_0 \), and then let \( U_1 \) be any open set such that \( C \subseteq U_1 \subseteq C'' \).
Then the formulas $\tilde{\varphi}_i$  given by $t\in  X\smallsetminus U_i$ are as required in~\ref{lem_interpolation-1}.
  
Suppose now that $\tilde{\varphi} \perp \varphi$ is of the form $t\in\tilde{C}$ for some closed set $\tilde{C}$ disjoint from $C$.
  By  normality of $X$, there is  a closed neighborhood $C'$ of $C$ which is disjoint from $\tilde{C}$.
Then the \( \FF \)-formula  \( \varphi' > \varphi \) given by \( t\in C' \) is such that \( \varphi' \perp \tilde{\varphi} \) and is as required in~\ref{lem_interpolation-2}.
  
The inductive cases are straightforward.
\end{proof}

For every type $p(x \,; \eta)\subseteq{\FF}(M)$, we let
\[
p'(x \, ; \eta) = \big\{\varphi'(x \, ; \eta) \MID \varphi'>\varphi \text{ for some } \varphi(x \, ; \eta)\in p(x \, ; \eta)\big\}.
\]

In particular,  $\{\varphi(x \, ; \eta)\}' = \big\{\varphi'(x\, ;\eta) \MID \varphi'>\varphi\big\}$.

 We say that a formula $\varphi(x\,;\eta)$ is \emph{approximately satisfiable}  in a  standard structure $M$ if the set $\{\varphi(x\,;\eta)\}'$ of its approximations is satisfiable in $M$. By adapting the   example  in  \cite[Chapter 5]{HI} to the current setting, one can show that, in general, approximate satisfiability does not imply satisfiability.

\section{Morphisms}\label{morphisms}

\def\ceq#1#2#3{\parbox[t]{35ex}{$\displaystyle #1$}\parbox{5ex}{\hfil $#2$}{$\displaystyle #3$}}

In this section we  introduce the notion of elementary map with respect to both the classes of $\HH$- and $\FF$-formulas, as well as their approximate versions, and we establish some basic facts about them.

\begin{definition} \label{def:elementarymaps}
Let $M$ and $N$ be  standard structures.
We say that a partial map $f \colon M\to N$ is \emph{${\FF}$-elementary} if for all $\varphi(x)\in{\FF}$  and all $a\in({\rm dom }f)^{|x|}$ 
\begin{equation} \tag{\( \dagger \)} \label{eq:dagger}
M \models \varphi(a) \quad \Rightarrow \quad N \models \varphi( f  a).
\end{equation}

 An \emph{${\FF}$-elementary embedding} is a total  ${\FF}$-elementary map. If \( M \subseteq N \) and the identity map ${\rm id}_M \colon M \to N$ is an ${\FF}$-elementary embedding we  write $M\preccurlyeq^{\FF}N$ and we say that $M$ is an \emph{${\FF}$-elementary substructure} of $N$, or that \( N \) is an \emph{\( \FF \)-elementary extension} of \( M \).
 
 The definitions of \emph{${\HH}$-elementary} map or \emph{\( \HH \)-elementary embedding} are  obtained by replacing ${\FF}$ with ${\HH}$ throughout.
 \end{definition}

Recall that the language $\LL_{\sf S}$ has a name for each element of $ X$. Hence, if  $f$ is ${\FF}$-elementary the implication \eqref{eq:dagger} actually holds for all $\varphi(x \, ; \eta)\in{\FF}$ in the following form:  for all $a\in({\rm dom }f)^{|x|}$ and all $\alpha\in  X^{|\eta|}$
$$M\models\varphi(a \, ; \alpha)\quad \Rightarrow \quad N\models\varphi(fa \,; \alpha).$$

The same remark applies to the classes of morphisms that we define next.

As ${\FF}$- and ${\HH}$-elementary maps are in particular $\LL_{\sf H}$-elementary, they are injective; however, the inverse of an ${\FF}$-elementary map  may not be ${\FF}$-elementary. 
This is because the class of $\FF$-formulas is not closed under negation, hence  from the validity  of~\eqref{eq:dagger}  for all ${\FF}$-formulas we cannot derive the converse implication for a specific \( \FF \)-formula \( \varphi \), unless \( C \) is clopen in all \( \FF \)-atomic subformulas of \( \varphi \) of the form \( t \in C \).

%
%


This notion of elementary map is dictated by our framework, in particular by our choice of working with the classical notion of satisfaction. This differs  from the approach in \cite{HI}, where \emph{approximate satisfaction} plays a crucial role. 

The approximate  ${\FF}$-morphisms  that we define next   preserve  the approximate satisfiability in \cite{HI}. They  are invertible in the class of approximate  ${\FF}$-morphisms, as discussed in Proposition~\ref{fact_HImorphisms} below.

\begin{definition} \label{def:approximateelementary}
We say that a partial map $f \colon M\to N$ is \emph{approximate ${\FF}$-elementary} if for every formula $\varphi(x)\in{\FF}$, and every $a\in({\rm dom }f)^{|x|}$
\begin{equation} \tag{\( \ddagger \)} \label{eq:ddagger}
M\models\varphi(a) \quad \Rightarrow \quad  N\models\big\{\varphi(fa)\big\}'.
\end{equation}

We define \emph{approximate ${\HH}$-elementary} maps in the same manner.
\end{definition}

Clearly,  ${\FF}$-elementarity implies its approximate version, and similarly for ${\HH}$-elementarity.

Later we will show that, under a weak saturation assumption,  the converse implication to~\eqref{eq:ddagger} holds as well  (Proposition~\ref{prop_approx}).
Furthermore, under full saturation, approximate elementary morphisms become ${\LL}$-elementary maps (Corollary~\ref{corol_Lcomplete}).

The following result holds without any  saturation assumption.

\begin{proposition}\label{fact_HImorphisms}
  If $f \colon M\to N$ is an approximate $\FF$-elementary map, then for every $\varphi(x)\in{\FF}$ and every $a\in({\rm dom }f)^{|x|}$,
\[
M\models\big\{\varphi(a)\big\}' \quad \Leftrightarrow \quad N\models\big\{\varphi(fa)\big\}'.
\]

The same holds for approximate ${\HH}$-elementary maps and \( \HH \)-formulas \( \varphi(x) \).
\end{proposition}

\begin{proof} 
%
%
  For the left-to-right implication,  assume  $M\models\big\{\varphi(a)\big\}'$. Given an arbitrary \( \varphi' > \varphi \), use density of \( > \) and pick any $\varphi'>\varphi''>\varphi$. Then \( M \models \varphi''(a) \) by assumption, hence \( N \models \varphi'(fa) \) by~\eqref{eq:ddagger} applied to \( \varphi'' \). Since $\varphi'$ was chosen arbitrarily, the conclusion follows.
  
  For the converse implication, assume $N\models\big\{\varphi(fa)\big\}'$ and
  let  $\varphi'>\varphi''>\varphi$. 
    By Lemma~\ref{lem_interpolation} there are $\tilde{\varphi}_0, \tilde{\varphi}_1 \perp\varphi''$ with \( \tilde{\varphi}_1 > \tilde{\varphi}_0 \) such that $\varphi'' \to \neg\tilde{\varphi}_1$ and \( \neg \tilde{\varphi}_0 \rightarrow\varphi' \).
    Since \( N \models \varphi ''(fa) \) by assumption, we get \( N \models \neg \tilde{\varphi}_1(fa) \), and thus \( M \models \neg \tilde{\varphi}_0(a) \) by~\eqref{eq:ddagger} applied to \( \tilde{\varphi}_0 \). Hence \( M \models \varphi'(a) \), and since \( \varphi' \) was an arbitrary approximation of \( \varphi \) we are done.
 \end{proof}

Finally, we define  the analogues of classical partial embeddings.

\begin{definition} \label{def:partialembedding}
A map $f \colon M\to N$ is a \emph{partial ${\FF}$-embedding} if~\eqref{eq:dagger} above holds for all ${\FF}$-atomic formulas; equivalently,~\eqref{eq:dagger} holds for all $\FF$-formulas $\varphi(x)$  without quantifiers of sort \( \sf H \).
\end{definition}

\begin{remark} \label{rmk:partialembedding}
  Let $f \colon M\to N$ be a partial ${\FF}$-embedding. Then, for every $a\in({\rm dom }f)^{|x|}$ and every ${\FF}$-formula $\varphi(x)$ without quantifiers of sort ${\sf H}$,
\[
 M\models\varphi(a) \quad \Leftrightarrow \quad N\models\varphi(fa).
 \]
 The forward implication is by definition of a partial \( \FF \)-embedding, so we only need to check the backward implication.
This is trivial for ${\FF}$-atomic formulas \( \varphi \in \LL_{\sf H} \), as they are closed under negations. For an \( \FF \)-atomic formula of the form $t(x)\in C$, assume
 $M\not\models t(a)\in C$. Since  $M\models t(a)\in\{c\}$ with $c=t^M(a) \notin C$, then \( N \models t(fa) \in \{ c \} \), and hence $N\not\models t(fa)\in C$.
The inductive cases are  straightforward. (In order to deal with quantifiers of sort \( \sf S \), recall  that the language $\LL_{\sf S}$ has a name for each element of $X$.)
\end{remark}

\section{The standard part of a structure}\label{standard_part}

Our goal is to prove the Compactness Theorem for the class of ${\FF}$-formulas. We  accomplish this in Section~\ref{compactness}. We anticipate that we do not resort to an ultraproduct construction. We rather build on the Compactness Theorem for classical logic and  make use of two basic notions in nonstandard topology,  those of monad of a point and of near-standard point.
In this section we prove the preliminary  Lemma~\ref{lem_st} that  we later use in the proof of Theorem~\ref{thm_compattezza}.

Fix  an $\LL_{\sf S}$-elementary extension ${}^*\! X$ of $ X$.
Here and below, we denote by ${}^*\!D$  the interpretation of the predicate symbol $D$ in $ {}^*\! X$.
%
%
Let $\eta$ be an $n$-tuple of variables of sort ${\sf S}$, for some \( n \geq 1 \).
For each $\beta\in  X^n$, we define the \( \LL_{\sf S} \)-type
\[
{\rm m}_\beta(\eta) = \{\eta\in D \MID  D \text{ is a closed neighborhood of }\beta\}.
\]

 Then   the  set $X_\beta$ of  realizations of ${\rm m}_\beta(\eta)$ in ${}^*\! X$ is 
\[
X_\beta=\bigcap\{ {}^*\!D \MID D \text{ is a closed neighborhood of }\beta\}.
\]

Thus, if we regard ${}^*\! X$  as a nonstandard extension of the topological space $X$ and recall that the closed neighborhoods of $\beta$ form a basis for its neighborhood system, we have that $X_\beta$ is just what, in nonstandard language, is called the \emph{monad} of $\beta$.


\begin{remark}\label{fact_uniqueness_st}
  For every $\alpha\in({}^*\! X)^n$ there is a unique $\beta\in  X^n$ such that ${}^*\! X\models{\rm m}_\beta(\alpha)$. This is straightforward if one uses the nonstandard characterization of compactness (see, for instance, \cite{D}) and the Hausdorff assumption on $X$. However,
for the benefit of the reader who is not familiar with this argument, we provide a self-contained model-theoretic argument. 
 The uniqueness of $\beta$ follows from the Hausdorff property of $ X^n$, so we just prove existence. 
For a contradiction, let  $\alpha\in({}^*\! X)^n$ be such that for each $\beta\in  X^n$ there exists a closed neighborhood $D_\beta$ of $\beta$ with ${}^*\! X \not \models \alpha \in D_\beta$. By compactness of $ X^n$, there is some finite subset $Y$ of $ X^n$ such that $ X^n=\bigcup_{\beta\in Y } D_\beta.$  From $ X\preccurlyeq {}^*\! X$ we get  $({}^*\! X)^n\subseteq\bigcup_{\beta\in Y} {}^*D_\beta$, contradicting the fact that $\alpha \notin {}^*D_\beta$ for all \( \beta \in X^n \) .
\end{remark}

We refer to Remark~\ref{fact_uniqueness_st} by saying that every point in $({}^*\! X)^n$ is near-standard. We call $\beta$ the \emph{standard part} of $\alpha$, and we denote it by ${\rm st}(\alpha)$.
  By definition of ${\rm st}(\alpha)$, if \( \alpha \in ({}^*\!  X)^n \)  and $C \subseteq X^n$ is a closed neighborhood of ${\rm st}(\alpha)$, then ${}^*\! X\models \alpha\in C$. The following proposition provides a converse to this, and it implies that if \( \alpha \in X^n \) then \( {\rm st}(\alpha) = \alpha \). (Take \( C = \{ \alpha \} \) and use the fact that \( X \preccurlyeq {}^*\! X  \).)

\begin{proposition}\label{fact_st1}
Let $\alpha\in({}^*\!  X)^n$ and let 
$C \subseteq X^n$ be a closed set. Then
\[
{}^*\! X \models \alpha\in C\ \to\ {\rm st}(\alpha)\in C.
\]
\end{proposition}

\begin{proof} 
By contraposition. Assume ${\rm st}(\alpha)\notin {}^*\!C$. From  $ X\preccurlyeq {}^*\! X$ we get that ${\rm st}(\alpha)\notin C.$
  By normality of $ X^n$ there is a closed neighborhood $D$ of ${\rm st}(\alpha)$ which is disjoint from $C$.  Hence ${}^*\!C\cap {}^*\!D=\emptyset$ by \( X \preccurlyeq {}^*\! X  \) again. On the other hand, by definition of   ${\rm st}(\alpha)$ we have $\alpha\in {}^*\!D$. Thus $\alpha\notin {}^*\!C$.
 \end{proof}


\begin{remark}\label{fact_tuple_st}
Let $\alpha=\langle\alpha_1,\dots,\alpha_n\rangle\in ({}^*\! X)^n$. Then 
\[
{\rm st}(\alpha)=\langle {\rm st}(\alpha_1),\dots, {\rm st}(\alpha_n)\rangle.
\] 
 Indeed, fix an arbitrary closed neighborhood \( D \) of \( \langle {\rm st}(\alpha_1),\dots, {\rm st}(\alpha_n)\rangle \). Then there are closed neighborhoods \( D_i \) of \( {\rm st}(\alpha_i) \) such that \( D_1 \times \ldots \times D_n \subseteq D \). By definition of \( {\rm st}(\alpha_i) \), we have \( \alpha_i \in {}^*\! D_i \) for all \( 1 \leq i \leq n \). Moreover, \( {}^*\! D_1 \times \ldots \times {}^*\! D_n \subseteq {}^*\! D \) by \( X \preccurlyeq {}^*\! X  \), hence \( \alpha \in {}^*\! D \).
 \end{remark}

 If we regard ${}^*\! X$  as a nonstandard extension of the topological space $ X$, the statement~\eqref{eq:circ} in the next proposition is a straightforward consequence of the nonstandard characterization of continuity in terms of monads (see~\cite{D}). 
 For the reader's convenience, below we provide a self-contained argument.

\begin{proposition}\label{fact_terms_st}
For all $\alpha\in({}^*\!  X)^{|\alpha|}$ and function symbols $f \in \LL_{\sf S}$ of sort ${\sf S}^{|\alpha|}\to{\sf S}$,
\begin{equation} \tag{\( \circ \)} \label{eq:circ}
{}^*\! X \models {\rm st}\big(f(\alpha)\big)=f\big({\rm st}(\alpha)\big).
\end{equation}
\end{proposition}

\begin{proof}  
 By definition of standard part,
it suffices to prove that ${}^*\! X\models f(\alpha)\in D$ for every closed neighborhood $D$ of $f\big({\rm st}(\alpha)\big)$.
  Fix one such $D$.
  By continuity 
 of $f$, the set  $C = f^{-1}[D]$ is a closed neighborhood of ${\rm st}(\alpha)$.
  By definition of \( {\rm st}(\alpha) \),
  ${}^*\! X\models \alpha\in C$.  Finally, $ X\preccurlyeq{}^*\! X$  yields that ${}^*\! X \models \forall \eta \, (\eta \in C \leftrightarrow f(\eta) \in D)$), hence ${}^*\! X\models f(\alpha)\in D$. 
\end{proof}

An elementary extension of a standard structure \( \langle M,X \rangle \) 
is a nonstandard \( \LL \)-structure of the form \( \langle N, {}^*\! X \rangle \) for some \( X \preccurlyeq {}^*\! X \). 
We now define a procedure to extract from any \( \LL \)-structure of the form \( \langle N, {}^*\! X \rangle \) a standard structure \( \langle N, X \rangle \) so that truth in \( \langle N, {}^*\! X \rangle \) is closely related to truth in \( \langle N, X \rangle \) (see Lemma~\ref{lem_st}).

%
%
%

\begin{definition} \label{def:standardpartofastructure}
Let \( \langle N, {}^*\! X \rangle \) be an \( \LL \)-structure with \( X \preccurlyeq {}^*\! X \). The \emph{standard part} of \( \langle N, {}^*\!X \rangle \) is the standard structure \( \langle N,X \rangle \) where
\begin{itemizenew}
\item
each symbol in \( \LL_{\sf S} \) is interpreted in \( X \) as usual;
\item
each symbol in \( \LL_{\sf H} \) is interpreted as in \( N \) (so these symbols have the same interpretation in both \( \langle N, {}^*\!X \rangle \) and \( \langle N,X \rangle \));
\item
each function symbol \( f \) of sort \( {\sf H}^n \to {\sf S} \) is canonically interpreted as the function \( f^{\langle N, X \rangle} \colon N^n \to X \) such that for all \( a \in N^n \)
\[ 
f^{\langle N,X \rangle}(a) = {\rm st} \big(f^{\langle N, {}^*\! X \rangle}(a)\big).
 \] 
\end{itemizenew}
\end{definition}
To simplify notation, for each function symbol \( f \in \LL \) we denote by \( {}^*\! f \) its interpretation \( f^{\langle N, {}^*\! X \rangle} \) in \( \langle N, {}^*\! X \rangle \), and by \( f^N \) its interpretation in the standard structure \( \langle N,X \rangle \) (which, as usual, will simply be denoted by \( N \) when there is no danger of confusion).  Similar notation will be adopted for \( \LL \)-terms.

\begin{proposition}\label{fact_st2} 
  With the above assumptions and  notation, let 
  $t(x\,;\eta)$ be an \( \LL \)-term of sort ${\sf H}^{|x|}\times{\sf S}^{|\eta|}\to {\sf S}$.
  Then, for all $a\in N^{|x|}$ and all $\alpha\in({}^*\! X)^{|\eta|}$
%
%
\begin{equation} \tag{\( \$ \)} \label{eq:dollar}
t^N\big(a\,;{\rm st}(\alpha)\big) = {\rm st}\big({}^*t(a\,;\alpha)\big).
\end{equation}
\end{proposition}

\begin{proof} 
The case when $t$ is a constant symbol or a variable of sort $\sf S$ is trivial.

If $t$ is $f(t_1,\dots, t_n)$ with $f$ of sort  ${\sf H}^n\to {\sf S}$,  then each $t_i$ is an $\LL_{\sf H}$-term and no variable of sort \( {\sf S} \) occurs in it. Thus  ${}^*t_i(a \, ; \alpha) = {}^*t_i(a)=t_i^N(a) = t_i^N(a \, ; {\rm st}(\alpha))$ for all $1\le i\le n$ by definition of \( \langle N,X \rangle \), and thus~\eqref{eq:dollar}  holds by definition of \( f^N \).

  If $t$ is $f(t_1,\dotsc, t_n)$, for some function symbol $f$ of sort ${\sf S}^n\to {\sf S}$,  inductively assume  that 
\[
t_i^N\big(a\,;{\rm st}(\alpha)\big) = {\rm st}\big({}^*t_i(a\,;\alpha)\big)
\]
holds for each term $t_i(x\,;\eta)$. Then the required equality  follows from 
Proposition~\ref{fact_terms_st}, Remark~\ref{fact_tuple_st}, and \( X \preccurlyeq {}^*\! X \).
%
\end{proof}


\begin{lemma}\label{lem_st}
  With the above assumptions and  notation,
  for every ${\FF}$-formula $\varphi(x\,;\eta)$, for  all $a\in N^{|x|}$ and all $\alpha\in({}^*\! X)^{|\eta|}$ 
 \begin{equation} \tag{\( + \)} \label{eq:plus}
 \langle N,{}^*\! X\rangle\models\varphi(a\,;\alpha) \quad \Rightarrow \quad \langle N, X \rangle \models \varphi\big(a\,;{\rm st}(\alpha)\big).
 \end{equation}
\end{lemma}

\begin{proof} 
By induction on the complexity of $\FF$-formulas. Suppose first that $\varphi(x\,;\eta)$ is \( \FF \)-atomic.
  If $\varphi(x\,;\eta)$ is an $\LL_{\sf H}$-formula, the validity of $(+)$ is straightforward.
  Otherwise, $\varphi(x\,;\eta)$  has the form $t(x\,;\eta)\in C$
   for some closed set \( C \subseteq X^n \) and \( t(x\,;\eta) \) a sequence of \( n \) terms \( t_i(x\, ; \eta) \) of sort \( {\sf H}^{|x|} \times {\sf S}^{|\eta|} \to {\sf S} \).
 Assuming  that $\langle N,{}^*\! X\rangle\models t(a\,;\alpha)\in C$ gives
 \( \langle N, {}^*\!X \rangle \models  {\rm st}\big({}^*t(a\,;\alpha)\big)\in C \) by Proposition~\ref{fact_st1}, and hence \( {\rm st}\big({}^*t(a\,;\alpha)\big)\in C \) by \( X \preccurlyeq {}^*\! X \).
  Therefore $t^N\big(a\,;{\rm st}(\alpha)\big)\in C$ by Remark~\ref{fact_tuple_st} and Proposition~\ref{fact_st2} applied to each term \( t_i \, \).
    The inductive cases are straightforward.
\end{proof}

There is a partial converse to Lemma~\ref{lem_st} using approximations. 

\begin{lemma} \label{lem_st2}
With the above assumptions and notations, for every \( \FF \)-formula \( \varphi(x \,; \eta) \), for all \( a \in N^{|x|} \) and all \( \alpha \in {({}^*\! X)}^{|\eta|} \)
\begin{equation} \tag{\( \times \)} \label{eq:times}
\langle N,X \rangle \models \varphi(a \,; {\rm st}(\alpha)) \quad \Rightarrow \quad \langle N, {}^*\! X \rangle \models \big\{ \varphi (a \,; \alpha) \big\}'.
\end{equation}
\end{lemma}

\begin{proof}
The proof is similar to that of Lemma~\ref{lem_st}, the only nontrivial case being that of an \( \FF \)-atomic formula \( \varphi \) of the form \( t(x \,; \eta) \in C \). Let \( \varphi' > \varphi \) be \( t(x \,; \eta) \in C' \), for some closed neighborhood \( C' \) of \( C \). If \( t^N(a \,; {\rm st}(\alpha)) \in C \), then 
by Proposition~\ref{fact_st2} and Remark~\ref{fact_tuple_st} the set  
\( C' \) is a closed neighborhood of \( {\rm st} \big( {}^* t(a \,; \alpha) \big) \). By definition of standard part, this means that \( {}^* t(a \,; \alpha) \in {}^* \! \big(C'\big) \), i.e.\ \( \langle N, {}^*\! X \rangle \models \varphi'(a \,; \alpha) \).
\end{proof}

\begin{corollary} \label{cor_st}
With the above assumptions and notations, for every \( \FF \)-formula \( \varphi(x \,; \eta) \), for all \( a \in N^{|x|} \) and all \( \alpha \in {({}^*\! X)}^{|\eta|} \)
\[ 
\langle N, {}^*\! X \rangle \models \big\{ \varphi (a \,; \alpha) \big\}' \quad \Leftrightarrow \quad \langle N,  X \rangle \models \big\{ \varphi (a \,; {\rm st}(\alpha)) \big\}'.
 \] 
\end{corollary}

\begin{proof}
The forward implication follows from Lemma~\ref{lem_st}. For the backward implication, fix any \( \varphi' > \varphi \) and by density of \( > \) pick any \( \varphi'' \) such that \( \varphi' > \varphi'' > \varphi \). Then apply Lemma~\ref{lem_st2} to \( \varphi'' \). \end{proof}


\section{Compactness}\label{compactness}

We distinguish between consistency with respect to standard structures and consistency with respect to arbitrary ${\LL}$-structures.
We say that an  {${\LL}$-theory $T$ is \emph{${\LL}$-consistent} if there is some ${\LL}$-structure $\langle M,J\rangle$   which satisfies $T.$
We say that $T$ is \emph{consistent} if $M\models T$ for some standard structure $M$.
For simplicity we consider only the class of formulas \( \FF \), but all definitions and results can easily be adapted to the subclass \( \HH \).

\begin{theorem}[\( \FF \)-Compactness Theorem]
\label{thm_compattezza}
  Let $T\subseteq{\FF}$ be finitely ${\LL}$-consistent.
  Then $T$ is consistent. 
\end{theorem}

\begin{proof}
  Suppose that $T$ is finitely consistent.   By the  Compactness Theorem for classical logic, there is some ${\LL}$-structure $\langle M,{}^*\! X\rangle$, with $ X\preccurlyeq {}^*\! X$, such that $\langle M,{}^*\! X\rangle\models T$.
  Let $\langle M, X\rangle$ be its standard part as in Definition~\ref{def:standardpartofastructure}.
  Then $\langle M, X\rangle\models T$ by Lemma~\ref{lem_st}.
\end{proof}

It is an open problem\footnote{An answer has been recently provided by~\cite{Z}.} whether it is possible to extend Theorem~\ref{thm_compattezza} to a larger class of formulas, in particular when  we allow  in ${\LL}$ function symbols of sort ${\sf H}^n\times {\sf S}^m\to {\sf S}$, with both $n$ and $m$ positive. Functions of this sort have a natural interpretation, for instance in the case of a group acting on a compact Hausdorff space.

\begin{definition}
A  standard structure $N$ is \emph{$\lambda$-${\FF}$-saturated} if it realizes all types 
 $p(x\,;\eta)\subseteq{\FF}(N)$ 
that are finitely realized in $N$ and that have fewer than $\lambda$ parameters from $N$.
When $\lambda=\vert N \vert $ we simply say that $N$ is  \emph{${\FF}$-saturated.}
\end{definition}

Under suitable set-theoretic assumptions, the existence of ${\FF}$-saturated standard structures can be proved as in the classical case.

\begin{theorem}\label{thm:monster}
  Assume \(\lambda\) is such that \( \lambda^{<\lambda} =  \lambda > |\LL| \).
Then every standard structure \( M \) of cardinality $\lambda$ has an ${\FF}$-elementary extension to an \( \FF \)-saturated standard structure of cardinality $\lambda$.
\end{theorem}

\begin{proof} 
  We need the following easy consequence of the $\FF$-compactness theorem.

\begin{itemize}
 \item[($\star$)]    For $\lambda$ and $M$ as in the assumptions, there is a model $M'$ of cardinality $\lambda$ that, for every $A \subseteq M$ with $|A| < \lambda$, realises every type $p(x,\eta)\subseteq\FF(A)$ that is finitely consistent in some $M''\succeq^\FF M$.
\end{itemize}

  Note that finite consistency of $p(x,\eta)$ \textit{in some $\FF$-elementary extension\/} of $M$ is a weaker requirement than finite consistency in $M$. The proof of ($\star$) is entirely similar to the standard proof of the classical version.
  
 The construction of the desired \( \FF \)-saturated standard structure is the usual one: we build an $\FF$-elementary chain $\langle M_i : i<\lambda\rangle$ of models of cardinality $\lambda$, where
 \begin{itemize}
 \item $M_0=M$
 \item we take unions at limit stages
 \item at stage $i+1$ we use $(\star)$ to get $M_{i+1}$.
 \end{itemize}
  Let $N$ be the union of the chain. Then $|N|=\lambda$ and, by an analogue of the classical elementary chain lemma, $N \succeq^\FF M_i$ for all $i$.
  Let $p(x,\eta)\subseteq \FF(A)$ be a type that is finitely consistent in $N$, where $|A|<\lambda$. Since $\lambda$ is regular, there is $i<\lambda$ such that $A\subseteq M_i$. By construction, $M_{i+1}$ realizes $p(x)$, and therefore, by $\FF$-elementarity, so does $N$.
\end{proof}



 With reference to the notion of approximate satisfiability that we introduced at the end of Section~\ref{ultrapws},  we show that $\omega$-${\FF}$-saturation implies that approximate satisfiability is equivalent to satisfiability. 

\begin{proposition}\label{prop_approx}
  Let $N$ be an $\omega$-${\FF}$-saturated standard structure. Then for every formula $\varphi(x; \eta)\in{\FF}(N)$, for all $a\in N^{\vert x\vert}$ and $\alpha\in  X^{\vert \eta\vert}$
\[
 N\models\big\{\varphi(a\,;\alpha)\big\}' \quad \Leftrightarrow \quad N\models\varphi(a\,;\alpha).
 \] 
 
Therefore, in \( N \) each \( \FF \)-type \( p(x\,; \eta) \) is equivalent to its approximation $p'(x\,; \eta)$.
\end{proposition}

\begin{proof}
The right-to-left implication is trivial because \( \varphi \to \varphi' \) for every \( \varphi' > \varphi \). 

The converse implication trivially holds for \( \FF \)-atomic formulas in \( \LL_{\sf H} \), as in this case \( \big \{ \varphi(a \,;\alpha) \big \} ' = \{ \varphi(a\,; \alpha) \} \). Suppose now that \( \varphi \) is an \( \FF \)-atomic formula of the form $t\in C$, where $C$ is a closed subset of $ X^n$. Since $ X^n$ is compact Hausdorff, and hence regular,
if $t^N(a\,;\alpha)\notin C$ then there is some closed neighborhood $C'$ of $C$ such that $t^N(a\,;\alpha)\notin C'$, which means that \( N \not\models \varphi'(a\,; \alpha) \) where \( \varphi' > \varphi \) is \( t \in C' \).

 The inductive steps  for conjunction, disjunction,
  and  for the universal quantifiers 
   of  sort $\sf H$ and \( \sf S \) are straightforward. 

%
   %
 %
%
%
%
%
%

We now treat the case  of the existential quantifier of sort ${\sf H}$. Consider the formula $\exists y \, \varphi(x, y \, ;\eta)$, and fix $a\in N^{\vert x\vert}$ and $\alpha\in  X^{\vert \eta\vert}$. By inductive hypothesis, for all $b\in N$
\begin{equation} \tag{\( \sim \)} \label{eq:sim}
N\models\big\{\varphi(a,b;\alpha)\big\}' \quad \Rightarrow  \quad N\models\varphi(a,b;\alpha).
\end{equation}
Assume that $N\models\{\exists y \, \varphi(a,y \,;\alpha)\}'$: we prove that the \( \FF \)-type $p(z)=\{\varphi'(a,z \, ;\alpha) \MID \varphi'>\varphi\}$  is realized in $N$ by some \( b \in N \), so that \( N \models \varphi(a,b\,; \alpha) \) by~\eqref{eq:sim}, and thus \( N \models \exists y \, \varphi(a,y \,; \alpha) \).
The \( \FF \)-type \( p \) uses only finitely many parameters from \( N \) because so does \( \varphi \), hence by \( \omega \)-\( \FF \)-saturation of \( N \) it is enough to show that \( p \) is finitely realizable in \( N \). 
For notational simplicity, we show that if $\varphi'_1,\varphi'_2>\varphi$,  then 
$(\varphi_1\land\varphi_2)(a,z \,; \alpha)$ is realized in $N$. 
First we claim that there is some $\varphi'>\varphi$ such that $\varphi'\rightarrow\varphi'_1\wedge\varphi'_2$. In order to obtain one such $\varphi'$, it suffices to replace each \( \FF \)-atomic subformula of $\varphi$ of the form $t\in C$ with $t\in C_1\cap C_2, $, where $t\in C_1$ and $t\in C_2$ are the \( \FF \)-atomic subformulas that occur in $\varphi'_1$ and $\varphi'_2$  in place of $t\in C$.
By assumption, $N\models\exists y \, \varphi'(a,y \,;\alpha)$, hence $N\models (\varphi_1\wedge\varphi_2)(a,b\, ;\alpha)$, for some $b\in N$, as desired.




The case of an existential quantifier of sort ${\sf S}$ is  similar. 
\end{proof}


\def\ceq#1#2#3{\parbox[t]{20ex}{$\displaystyle #1$}\parbox{5ex}{\hfil $#2$}{$\displaystyle #3$}}


A consequence of Proposition~\ref{prop_approx} is that approximate \( \FF \)-elementary maps (see Definition~\ref{def:approximateelementary})  between $\omega$-${\FF}$-saturated standard structures coincide with their unapproximated version.
In particular, Proposition~\ref{fact_HImorphisms} gives the following corollary.

\begin{corollary}\label{corol_omega_sat}
  Let $M$ be an $\omega$-${\FF}$-saturated standard structure and $f \colon M\to N$ be an ${\FF}$-elementary map.
  Then
\[
M\models\varphi(a) \quad \Leftrightarrow \quad N\models\varphi(fa)
\]
for every ${\FF}$-formula $\varphi(x)$ and every $a\in({\rm dom }f)^{|x|}$.
If moreover \( N \) is \( \omega \)-\( \FF \)-saturated as well, then the same is true for every approximate \( \FF \)-elementary map \( f \colon M \to N \).
\end{corollary}

\section{The monster model}\label{monster}

Throughout this section we fix  a sufficiently large ${\FF}$-saturated standard structure $\langle  \EuScript{U},  X\rangle$, which exists by Theorem~\ref{thm:monster}, and as usual  identify it with its first component \( \EuScript U \).  (Recall that the structure $ X$ is fixed).  We call $\EuScript{U}$ the \emph{monster model}. The  truth value of a formula is evaluated in  $\EuScript{U}$, unless otherwise specified. An ${\FF}$-elementary substructure of $\EuScript{U}$ is called a \emph{model}. From now on, we  assume that all standard structures under consideration are models.

\begin{remark}\label{remark_identity} 
For all models $M, N$ and all $a\in M^{\vert a\vert}\cap N^{\vert a\vert}$,  the identity map on the elements of the tuple $a$ is an approximate $\FF$-elementary map from $M$ to $N$. Indeed, let $\varphi$ be an $\FF$-formula such that 
\( M \models \varphi(a) \), so that \( \EuScript{U} \models \varphi(a) \) because \( M \) is a model, and assume towards a contradiction that
$N\not\models \varphi'(a)$ for some $\varphi'>\varphi$. Let $\varphi''$ be such that $\varphi'>\varphi''>\varphi$, so that $N\not\models \varphi''(a)$. 
By Lemma~\ref{lem_interpolation}\ref{lem_interpolation-1} there exists 
$\tilde{\varphi}\perp\varphi$ such that 
$\varphi\rightarrow\neg\tilde{\varphi}\rightarrow\varphi''$. Since $N\not\models \varphi''(a)$ then  $N\models\tilde{\varphi}(a)$. Thus $\EuScript{U}\models \tilde{\varphi}(a)$ because \( N \) is a model, a contradiction.
\end{remark}

By definition of \( \FF \)-elementary embedding,
the truth of an ${\FF}(M)$-formula is preserved 
from a model $M$ to the monster  model;
if the converse implication also holds we say that $M$ is a \emph{strong model}. 
By Corollary~\ref{corol_omega_sat}, any \( \omega \)-\( \FF \)-saturated model is strong.
Notice also that any (not necessarily $\omega$-$\FF$-saturated)  strong model $M$ satisfies the conclusion of Proposition~\ref{prop_approx}. Furthermore, by Remark~\ref{remark_identity} 
every model that satisfies the conclusion of Proposition~\ref{prop_approx} is a strong model.  Hence the strong models are precisely the  models for which the equivalence in Proposition~\ref{prop_approx} holds. 

For the rest of this section,
we fix a small subset $A$ of $\EuScript{U}$.
For all $n \geq 1$ we denote by $\tau_n(A)$ the topology on $\EuScript{U}^{n}$ for which the closed sets are those defined by the $n$-types $p(x)\subseteq{\FF}(A)$.
By \( \FF \)-saturation the topology $\tau_n(A)$ is compact. 
 
In what follows, if $p(x)\subseteq \FF(A)$ and $\psi(x)$ is an $\FF(A)$-formula, we write $p(x)\rightarrow \psi(x)$ 
when $p(\EuScript{U}) \subseteq \psi(\EuScript{U})$. When working in the monster model $\EuScript{U}$, we freely use infinitary connectives in similar expressions involving types --- the meaning should be clear from the context.

Next we prove two consequences of  ${\FF}$-compactness that hold in our setting.
Note the   differences with  the classical setting. 

\begin{proposition}\label{fact_compactness_imp}
  Let $p(x)\subseteq{\FF}(A)$ and  $\varphi(x)\in{\FF}({A})$.
  \begin{enumerate-(1)}
    \item \label{fact_compactness_imp-1}
    If $p(x)\rightarrow\neg\varphi(x)$, 
    then $\psi(x)\rightarrow\neg\varphi(x)$ for some finite conjunction $\psi(x)$ of formulas in $p(x)$.
    \item \label{fact_compactness_imp-2}
    If $p(x)\rightarrow\varphi(x)$ and $\varphi'>\varphi$, then $\psi(x)\rightarrow\varphi'(x)$ for some finite conjunction $\psi(x)$ of formulas in $p(x)$.
  \end{enumerate-(1)} 
\end{proposition}

\begin{proof}
Part~\ref{fact_compactness_imp-1} is a straightforward consequence of $\FF$-compactness (Theorem~\ref{thm_compattezza}). Part~\ref{fact_compactness_imp-2} follows from Lemma~\ref{lem_interpolation}\ref{lem_interpolation-1} and part~\ref{fact_compactness_imp-1}.
\end{proof}

For a tuple $x$ of variables, ${{\FF}_x}(A)$ denotes the set of $\FF$-formulas whose free variables are in $x$.  We use similar notation for ${\HH}$-formulas and \( \LL \)-formulas.  

Let $a$ be an $n$-tuple of elements from $\EuScript{U}$. We let 
\[
\FF\mbox{-tp}(a/A)=\big\{ \varphi(x)\in{\FF}(A) \MID \EuScript{U}\models\varphi(a) \big\}, \qquad S_{\FF,n}(A)=\big\{\FF\mbox{-tp}(a/A) \MID a\in \EuScript{U}^n\big\}, 
\]
and $ S_{\FF}(A)=\bigcup_{n \in \omega} S_{\FF,n}(A)$. If \( x \) is a tuple of variables, \( S_{\FF,x}(A) \) denotes the subset of \( S_\FF(A) \) consisting of those types whose free variables range in \( x \). 


 \begin{proposition}\label{prop_Hcomplete2}
   Each type $p(x)\in S_{\FF,x}(A)$ is a maximal consistent subset of ${\FF}_x(A)$ in the following sense: for every $\varphi(x)\in{\FF}_x(A)$, either $\varphi(x)\in p(x)$ or%
\footnote{The asymmetric definition of maximality is due to the fact that \( \FF \) is not closed under negations, so we cannot require \( \neg \varphi(x) \in p(x) \).} 
$p(x)\rightarrow\neg\varphi(x)$. 

The same holds for the types in $S_{\HH,x}(A)$ (in which case maximality refers to \( \HH_x(A) \)).
 \end{proposition}

 \begin{proof}
 Let $p(x)={\FF}\mbox{-tp}(a/A)$, and pick any \( \varphi(x) \in \FF(A) \). If \( \varphi(x) \notin p(x) \) (i.e.\ \( \neg \varphi(a) \)), then there is \( \varphi' > \varphi \) such that \( \neg \varphi'(a) \) by Proposition~\ref{prop_approx}. 
By Lemma~\ref{lem_interpolation}\ref{lem_interpolation-1} there is \( \tilde{\varphi} \perp \varphi \) such that \( \varphi \to \neg \tilde{\varphi} \to \varphi' \). Thus \( \tilde{\varphi}(a) \), and so \( p(x) \to \tilde{\varphi}(x) \) because \( \tilde{\varphi}(x) \in p(x) \). Since \( \tilde{\varphi} \to \neg \varphi \), it follows that \( p(x) \to \neg \varphi(x) \), as desired.
 \end{proof}

\subsection{Approximate elimination of quantifiers of sort \textsf{S}}

In this subsection we show that the quantifiers of sort $S$ can be eliminated at the cost of allowing approximations.  In a classical setting, proving that every complete type over a theory $T$ is equivalent to its quantifier-free part amounts to proving quantifier elimination for $T$. Here we prove that, in our setting, every complete ${\FF}$-type is equivalent to its  restriction to ${\HH}$. However, this equivalence does not imply that every ${\FF}$-formula is equivalent to an ${\HH}$-formula: given $\varphi \in \FF$  and \( \varphi'> \varphi \), we can find $\psi \in \HH$ such that $\varphi \rightarrow \psi$, but we can only obtain \( \psi \to \varphi ' \) as a converse implication.

We say that a standard structure $M$ is \emph{${\HH}$-homogeneous} if every ${\HH}$-elementary map $f \colon M\to M$ of cardinality $<\vert M\vert$ extends to an automorphism of \( M \). Notice that if \( M \) is \( \omega \)-\( \FF \)-saturated, then one can equivalently consider \emph{approximate} \( \HH \)-elementary maps in the above definition. One could also define a notion of \( \FF \)-homogeneity: however, since \( \HH \subseteq\FF \) the latter would be a weaker notion. The following result is thus the strongest we can expect.

\begin{proposition}\label{prop_homog}
  The monster model (indeed, any $\FF$-saturated standard structure) is ${\HH}$-homogeneous.
\end{proposition}

\begin{proof} 
%

Let $f$ be a partial ${\HH}$-elementary map such that $\vert f\vert<\vert \EuScript{U}\vert$. We use a standard back-and-forth construction to extend \( f \) while maintaining \( \HH \)-elementarity, until we get a total surjective map, i.e.\ an automorphism of \( M \).

We first show how to add a new element \( b \in \EuScript{U} \) to the domain of \( f \). Let \( B ={\rm dom} f \)
and let $p(x/B)= {\HH}\mbox{-tp}(b/B)$. By \( \HH \)-elementarity, the conjugate $\HH$-type $p(x/f(B))$ is finitely satisfiable in \( \EuScript{U} \), hence it has a realisation $c$ by saturation of $\EuScript{U}$. Then $f \cup \{\langle b,c \rangle\}$ is an ${\HH}$-elementary extension of $f$.

To add an element to the range of \( f \), we apply the same argument to \( f^{-1}\). This is possible because, as observed in Section~\ref{morphisms}, the map $f$ is injective, hence invertible, and its inverse \( f^{-1} \) is \( \HH \)-elementary by saturation and Propositions~\ref{fact_HImorphisms} and~\ref{prop_approx}.
\end{proof}

We now strengthen Proposition~\ref{prop_Hcomplete2} {by showing that consistent \( \HH \)-types are already complete with respect to the class of all \( \LL_x(A) \)-formulas.
Recall that, throughout this section, $A\subseteq \EuScript{U}$ is a fixed set with \( |A| < |\EuScript{U} |\).

\begin{corollary}\label{corol_Lcomplete}
  Let $p(x)\in S_{\HH,x}(A)$.
  Then $p(x)$ is complete for  the formulas in ${\LL}_x(A)$.
  Namely, for every $\varphi(x)\in{\LL}_x(A)$, either $p(x)\rightarrow\varphi(x)$ or $p(x)\rightarrow\neg\varphi(x)$.
  
 Furthermore, $p'(x)$ is also complete for  the formulas in ${\LL}_x(A)$.
 \end{corollary}

\begin{proof}
  If $b\in \EuScript{U}$ is an arbitrary realization of ${\HH}\mbox{-tp}(a/A)$, 
 then the map \( f \) which is the identity on \( A \) and sends \( a \) to \( b \) is \( \HH \)-elementary.
  By Proposition~\ref{prop_homog}, such $f$ exends to an automorphism of $\EuScript{U}$.
  Since every automorphism is ${\LL}$-elementary, the conclusion follows.
  
  The last part of the statement follows from Proposition~\ref{prop_approx}.
  \end{proof}

Let $a,b\in \EuScript{U}^{n}$. We write $a\equiv_Ab$ if $a$ and $b$ satisfy the same ${\LL}(A)$-formulas in $n$  free variables. 
By Corollary~\ref{corol_Lcomplete}, formulas in ${\HH}(A)$ or ${\FF}(A)$, or their approximations, suffice to test $\equiv_A$.

The next proposition shows that formulas in ${\FF}(A)$ can be approximated by formulas in ${\HH}(A)$.

\begin{proposition}\label{prop_LHapprox1}
  Let $\varphi(x)\in{\FF}(A)$.
  For every  $\varphi'>\varphi$ there is some formula $\psi(x)\in{\HH}(A)$ such that $\varphi(x)\rightarrow\psi(x)\rightarrow\varphi'(x)$. 
\end{proposition}

\begin{proof} 
By Corollary~\ref{corol_Lcomplete}
\[
\neg\varphi(x)\rightarrow\bigvee \big \{ p(x) \MID p(x) \in S_{\HH,x}(A) \text{ and }p(x)\rightarrow\neg\varphi(x) \Big\},
\]
that is, for all \( a \in \EuScript{U}^{|x|} \), if \( \neg \varphi(a) \) then \( p(a) \) for some \( p(x) \in S_{\HH,x}(A) \) with \( p(x) \to \neg \varphi(x) \).

%
%

 For each \( p(x) \in S_{\HH,x}(A) \) such that \( p(x) \to \neg \varphi(x) \) we also have \( p'(x) \to \neg \varphi(x) \) by Proposition~\ref{prop_approx}, hence by Proposition~\ref{fact_compactness_imp}\ref{fact_compactness_imp-1} applied to \( p'(x) \) there is an approximation \( \chi'_p(x) \) of some formula $\chi_p\in p(x)$ such that $\chi'_p(x)\rightarrow\neg\varphi(x)$. (Here we use the fact that \( p(x) \), and hence also   $p'(x)$, is closed under finite conjunctions.)

By applying Lemma~\ref{lem_interpolation}\ref{lem_interpolation-1} to $\chi_p<\chi_p'$ we get  a formula $\psi_p$ such that $\psi_p\perp\chi_p$ and $\chi_p\rightarrow\neg\psi_p\rightarrow\chi'_p$. 
Note that \( \psi_p\perp\chi_p \) implies that \( \psi_p \) is an \( \HH(A) \)-formula because so is \( \chi_p \).

Hence
\[
\neg\varphi(x) \rightarrow \bigvee\big\{\neg\psi(x) \MID  \psi(x)\in\HH(A) \text{ and } \neg{\psi}(x)\rightarrow\neg\varphi(x)\}
\]
or, equivalently,
\( q(x) \to \varphi(x) \), where \( q (x) \subseteq \HH(A) \subseteq \FF(A) \) is the type 
\[
q(x) = \big\{\psi(x) \MID  \psi(x)\in \HH(A) \text{ and } {\varphi(x)}\rightarrow{\psi}(x)\}.
\]
By applying Proposition~\ref{fact_compactness_imp}\ref{fact_compactness_imp-2}  
to \( q(x) \),
we get that for every $\varphi'>\varphi$ there are finitely many ${\psi}_1(x), \dotsc, \psi_n(x) \in q(x)$ such that \( \bigwedge_{1 \leq i \leq n}{\psi}_i(x) \to \varphi'(x) \). But since \( \psi_i(x) \in q(x) \) we also have \( \varphi(x) \to \psi_i(x) \) for every $1\leq i\leq n$, thus
\[
\varphi(x) \to \bigwedge_{1 \leq i \leq n}{\psi}_i(x) \to \varphi'(x)
\]  
and the middle conjunction  is
 the required \( \HH(A) \)-formula \( \psi(x) \).
\end{proof}

We also prove an approximation result for the negation of formulas in ${\FF}(A)$.

\begin{proposition}\label{prop_LHapprox2}
  Let $\varphi(x)\in{\FF}(A)$ be such that $\neg\varphi(x)$ is consistent.
  Then there is a consistent formula  $\psi(x)\in{\HH}(A)$ such that  $\psi'(x)\rightarrow\neg\varphi(x)$ for some $\psi'>\psi$.
\end{proposition}


\begin{proof}
  Let $a\in \EuScript{U}^{|x|}$ be such that $\neg\varphi(a)$, and let $p(x)={\HH}\mbox{-tp}(a/A)$.
  By Corollary~\ref{corol_Lcomplete}, $p'(x)\rightarrow\neg\varphi(x)$. 
 By Proposition~\ref{fact_compactness_imp}\ref{fact_compactness_imp-1} applied to \( p' (x) \), there is some $\psi'(x) \in p(x)$  such that $\psi'(x)\rightarrow\neg\varphi(x)$.  (Here we use again that \( p'(x) \) is closed under finite conjunctions.)
 By definition of \( p'(x) \),  \( \psi' > \psi \) for some \( \psi(x) \in p(x) \subseteq \HH(A) \), and \( \psi(x) \) is consistent because it is realized by \( a \).
\end{proof}

\subsection{The Tarski-Vaught Test and the L\"owenheim-Skolem Theorem}

The following proposition can be regarded as an analogue of the Tarski-Vaught Test  when the 
larger  structure is the monster model.

\begin{proposition}\label{prop_Tarski_Vaught}
  The following are equivalent for a subset $M$ of $\EuScript{U}$.
  \begin{enumerate-(1)}
    \item \label{prop_Tarski_Vaught-1}
     $M$ is the domain of a model, i.e.\ \( M \) is an \( \FF \)-elementary standard substructure of \( \EuScript{U} \) when equipped with the restriction of the interpretations in \( \EuScript{U} \) of the symbols of \( \LL \);
    \item \label{prop_Tarski_Vaught-2}
    for every formula $\varphi(x)\in{\HH}(M)$  without free variables of sort $\sf S$,
    \[
    \exists x \, \varphi(x) \quad \Rightarrow \quad \text{ for every }\varphi'>\varphi\text{ there is  } a\in M \text{ such that }\varphi'(a);
    \]
    \item \label{prop_Tarski_Vaught-3}
    for every formula $\varphi(x)\in{\FF}(M)$ without free variables of sort $\sf S$,
    \[
    \exists x \, \neg \varphi(x) \quad \Rightarrow \quad \text{ there is }a\in M\text{ such that }\neg\varphi(a).
    \]
\end{enumerate-(1)}
\end{proposition}

\begin{proof}


\ref{prop_Tarski_Vaught-1}~$\Rightarrow$~\ref{prop_Tarski_Vaught-2}.
Assume  that $\exists x\,\varphi(x)$,
and let $\varphi'>\varphi$.
  By Lemma~\ref{lem_interpolation}\ref{lem_interpolation-1} there is some $\tilde{\varphi}\perp\varphi$ such that  $\varphi(x)\rightarrow\neg\tilde{\varphi}(x)\rightarrow\varphi'(x)$.
  Then $\neg\forall x\,\tilde{\varphi}(x)$  and so, by part~\ref{prop_Tarski_Vaught-1}, $M\models\neg\forall x\,\tilde{\varphi}(x)$.
  Then $M\models\neg\tilde{\varphi}(a)$ for some $a\in M$. Hence $M\models\varphi'(a)$, and $\varphi'(a)$ follows by~\ref{prop_Tarski_Vaught-1} again.

\ref{prop_Tarski_Vaught-2}~\( \Rightarrow \)~\ref{prop_Tarski_Vaught-3}.
  Assume 
  that $\exists x\,\neg\varphi(x)$.
  By Proposition~\ref{prop_LHapprox2}, there are a consistent formula $\psi(x)\in{\HH}(M)$ and an  approximation $\psi'$ of $\psi$ such that $\psi'(x)\rightarrow\neg\varphi(x)$.
  Since \( \exists x \, \psi(x) \),
  by~\ref{prop_Tarski_Vaught-2}  there is \( a \in M \) such that \( \psi'(a) \), hence \( \neg \varphi(a) \) and~\ref{prop_Tarski_Vaught-3} is verified.

\ref{prop_Tarski_Vaught-3}~\( \Rightarrow \)~\ref{prop_Tarski_Vaught-1}.
 Recall that 
 \( \LL_{\sf H}(M) \subseteq {\HH}(M) \).
 Then by~\ref{prop_Tarski_Vaught-3} it is easy to see that $M$ is closed under the interpretations of the constant and the function symbols in $\LL_{\sf H}(M)$.
  Hence $M$ is the domain of an ${\LL_{\sf H}}$-substructure of $\EuScript{U}$.  
  By~\ref{prop_Tarski_Vaught-3} and the classical Tarski-Vaught test,   $M\preceq^{\LL_{\sf H}} \EuScript{U}$.  
 This also implies that \( \langle M,X \rangle \) is an \( \LL \)-substructure of \( \langle \EuScript{U},X \rangle \), which means that \( t^M(a) = t(a) \) for all \( \LL \)-terms \( t(x) \) and \( a \in M^{|x|} \). Therefore, 
  for every 
${\FF}$-atomic formula  $\varphi(x)$  and   every $a\in M^{|x|}$
 \[
M\models \varphi(a) \quad \Leftrightarrow \quad \varphi(a).
 \]

 For the inductive step, consider the ${\FF}$-formula $\forall y\, \varphi(x,y)$. The inductive assumption is that,  for all $a\in M^{|x|}$ and all $b\in M^{|y|}$, if
  $ {M\models\varphi(a,b)}$ then ${\varphi(a,b)}$.
We prove by contraposition that
if $M\models\forall y\,\varphi(a,y)$ then $\forall y\,\varphi(a,y)$.
  Indeed, using~\ref{prop_Tarski_Vaught-3} in the second step gives
\begin{align*}
\neg\forall y\,\varphi(a,y) \quad & \Rightarrow \quad \exists y\,\neg\varphi(a,y) & \\
& \Rightarrow \quad \neg\varphi(a,b) & \text{for some } b\in M^{|y|}\\
& \Rightarrow \quad M\not\models\varphi(a,b) & \text{for some } b\in M^{|y|}\\
& \Rightarrow \quad M\not\models\forall y\,\varphi(a,y). &
\end{align*}
%
%
%

  The inductive steps for the connectives $\vee$, $\wedge$ and the quantifiers $\existsH$, $\existsS$, and $\forallS$ are straightforward.
 \end{proof}

 In the classical setting, an important application of the Tarski-Vaught Test is in the proof of the Downward L\"owenheim-Skolem Theorem.
 In our setting, the latter holds  for arbitrary ${\EuScript L}$-structures;
 in particular, every subset $A$ of ${\EuScript U}$ is contained in a model of cardinality at most $\vert{\EuScript L}(A)\vert$.
However, in the latter form the L\"owenheim-Skolem Theorem is not very informative
 because of the potentially large cardinality of $\LL_{\sf S}$.
To prove a more meaningful result, we introduce the notion of \emph{separable} language.

 \begin{definition}\label{def_sep}
 A  language ${\EuScript L}$ is \emph{separable} if there is a countable subset $\FF_0$ of ${\EuScript F}$ such that for all $\varphi \in \FF$ and all \( \varphi'>\varphi \) there exists $\varphi_0 \in \FF_0$ such that $\varphi\rightarrow\varphi_0\rightarrow\varphi'$. 
 \end{definition}

 \begin{proposition}
   Assume that ${\EuScript L}$ is a separable language.
   Let $A\subseteq {\EuScript U}$ be countable.
   Then there is a countable model $M$ containing $A$. 
 \end{proposition}
 
 \begin{proof}
   As in the classical proof of the L\"owenheim-Skolem Theorem,
   construct a countable subset $M$ of ${\EuScript U}$ that contains \( A \) and a realization of every consistent formula in ${\FF}_0(M)$, where \( \FF_0 \) witnesses the separability of \( \LL \).
   By the equivalence \ref{prop_Tarski_Vaught-1}\( \iff \)\ref{prop_Tarski_Vaught-2} in Proposition~\ref{prop_Tarski_Vaught}, $M$ is a model.
 \end{proof}

The following result provides a sufficient condition for the separability of ${\EuScript L}$.

\begin{proposition}
Let \( X \) be second-countable. If \( {\EuScript L} \setminus {\EuScript L}_{\sf S} \) is countable, then \( {\EuScript L} \) is separable.
\end{proposition}

\begin{proof}
For each \( n \geq 1 \) fix a countable basis \( {\EuScript B}_n \) for \( X^n \), and let \( {\EuScript D}_n \) be the collection of those sets which are the closure of a finite union of elements of \( {\EuScript B}_n \). Set also \( {\EuScript D} = \bigcup_{n \geq 1 } {\EuScript D}_n \). For every closed \( C \subseteq X^n \) and every open set \( U \supseteq C \) there is a closed neighborhood \( D \in {\EuScript D}_n \) of \( C \) such that \( D \subseteq U \). Indeed, by normality of \( X \) there is a closed neighborhood \( E \) of \( C \) such that \( E \subseteq U \). By compactness of \( C \), there is a finite subfamily of \( {\EuScript B}_n \) whose union \( V \) satisfies \( C \subseteq V \subseteq E \). The closure \( D \) of \( V \) is as desired. 

Since \( X \) is Hausdorff, second-countable and regular, it is metrizable by the Urysohn Metrization Theorem, and hence completely metrizable because of compactness. Thus for $k\ge1$ the space \( C(X^k,X) \) of (uniformly) continuous functions \( f \colon X^k \to X \) is a completely metrizable separable (i.e.\ Polish) space when equipped with the topology generated by the neighborhoods of \( f \in C(X^k,X) \) given by
\[ 
\{g \in C(X^k,X) \mid \forall x \in X^k \, d(f(x),g(x)) < \varepsilon \},
 \] 
where \( \varepsilon \) varies over positive reals and \( d \) is any fixed (complete) metric on \( X \) compatible with its topology. 
(See e.g.\ \cite[Theorem 4.19]{Kec}.) 
Let \( {\EuScript G}_k = \{ f^{(k)}_n \mid n \in \omega \} \) be dense in \( C(X^k,X) \) and let ${\EuScript G}=\bigcup_{k\ge 1}{\EuScript G}_k$.

Let \( {\EuScript F}_0 \subseteq {\EuScript H} \) be the set of \( {\EuScript F} \)-formulas \( \varphi \) such that:
\begin{enumerate}
\item \label{separabilityterm}
if \( t \) is a term occurring in \( \varphi \), then
all functions of sort \( {\sf S}^k \to {\sf S} \) appearing in \( t \) belong to \( {\EuScript G} \);
\item
if an atomic formula \( t \in D \) occurs in \( \varphi \), then \( D \in {\EuScript D} \).
\end{enumerate}

We claim that \( {\EuScript F}_0 \) witnesses that \( {\EuScript L} \) is separable.
Clearly, \( {\EuScript F}_0 \) is countable because so are \( {\EuScript G} \), \( {\EuScript D} \), and \( {\EuScript L} \setminus {\EuScript L}_{\sf S} \). 

\begin{claim} \label{claim1}
For every term \( t(x\, ;\eta) \) of sort \( {\sf H}^{|x|} \times {\sf S}^{|\eta|} \to {\sf S} \) and every positive real \( \varepsilon \) there is a term \( t'(x\, ; \eta) \) of the same sort satisfying~\eqref{separabilityterm} and such that \( d(t(a\, ;\alpha),t'(a\, ;\alpha)) < \varepsilon \) for any standard structure \( M \),  \( a \in M^{|x|} \), and \( \alpha \in X^{|\eta|} \).
\end{claim}

\begin{proof}[of the claim]
By induction on the complexity of \( t \). The only nontrivial case is when \( t \) is of the form \( f(t_1, \dotsc, t_n) \) with \( f \) of sort \( {\sf S}^n \to {\sf S} \), so that each \( t_i \) is of sort \( {\sf H}^{|x|} \times {\sf S}^{|\eta|} \to {\sf S} \). Since \( f \) is uniformly continuous, there is \( \delta \) such that for all \( \alpha, \beta \in X^n \), if \( d_n(\alpha,\beta) < \delta \), where \( d_n \) denotes the sup metric on the product \( X^n  \), then \( d(f(\alpha),f(\beta)) < \varepsilon/2 \). By inductive hypothesis, for each \( t_i \) choose \( t'_i \) satisfying~\eqref{separabilityterm}
and such that \( d(t_i(a\, ;\alpha),t'_i(a\, ;\alpha)) < \delta \) for any standard structure \( M \), \( a \in M^{|x|} \), and \( \alpha \in X^{|\eta|} \). Let also \( f' \in {\EuScript G} \) be of sort \( {\sf S}^n \to {\sf S} \) and such that \( d(f(\alpha),f'(\alpha)) < \varepsilon/2 \) for all \( \alpha \in X^n \). Finally, let \( t'' \) be the term \( f(t'_1, \dotsc, t'_n) \) and \( t' \) be the term \( f'(t'_1, \dotsc, t'_n ) \). Clearly, \( t' \) satisfies~\eqref{separabilityterm}. Let \( M \) be a standard structure, and fix any \( a \in M^{|x|} \) and \( \alpha \in X^{|\eta|} \). Then by the choice of the \( t'_i \) and of \( \delta \), we have \( d(t(a\, ;\alpha),t''(a\, ;\alpha)) < \varepsilon/2 \). On the other hand, by the choice of \( f' \) we also have \( d(t''(a\, ;\alpha),t'(a\, ;\alpha)) < \varepsilon/2 \). Thus \( d(t(a\, ;\alpha),t'(a\, ; \alpha)) < \varepsilon \) by the triangle inequality, and we are done.
\end{proof}

Given a closed set \( C \subseteq X^n \) and a positive real \( \varepsilon \), the \( \varepsilon \)-enlargement \( C_\varepsilon \) of \( C \) is the open set constituted by the union of all the open \( d_n \)-balls with radius \( \varepsilon \) and center in \( C \).

\begin{claim}\label{claim2}
For every closed set \( C \subseteq X^n \) and every closed neighborhood \( D \) of \( C \) there is some positive real \( \varepsilon \) such that \( C_\varepsilon \subseteq D \).
\end{claim}

\begin{proof}[of the claim]
Suppose towards a contradiction that for every \( m \in \omega \) there is \( \alpha_m \in C \) and \( \beta_m \in X^n \setminus D \) such that \( d_n(\alpha_m, \beta_m) < 1/m \). Since \( C \) is compact, by passing to a subsequence we can assume that the points \( \alpha_m \) converge to some \( \alpha \in C \). But then there is no open neighborhood of \( \alpha \) contained in \( D \), contradicting the fact that \( D \) is a closed neighborhood of \( C \). Indeed, for every positive real \( \delta \) one can pick \( m \) large enough so that \( d_n(\alpha_m,\alpha) < \delta/2 \) and \( 1/m < \delta/2 \). Then \( d_n(\beta_m,\alpha) \leq d_n(\beta_m,\alpha_m)+d_n(\alpha_m,\alpha) < \delta \), and so the ball centered in \( \alpha \) with radius \( \delta \) is not contained in \( D \).
\end{proof}

Fix now arbitrary formulas \( \varphi \in {\EuScript F} \) and \( \varphi' > \varphi \). By induction on the complexity of \( \varphi \), we show that there is \( \varphi_0 \in {\EuScript F}_0 \) such that \( \varphi \to \varphi_0 \to \varphi' \).

If \( \varphi \) is an \( {\EuScript L}_{\sf H} \)-formula, then we can let \( \varphi_0 \) be \( \varphi \) itself. Assume now that \( \varphi \) is of the form \( t(x\, ;\eta) \in C \), for \( C \subseteq X^n \) closed and \( t(x \, ; \eta) \) an \( n \)-tuple of terms \( t_i(x\, ; \eta) \) each of sort \( {\sf H}^{|x|} \times {\sf S}^{|\eta|} \to {\sf S} \). The approximation \( \varphi' \) is then of the form \( t(x\, ;\eta) \in D \), for some closed neighborhood \( D \) of \( C \). By Claim~\ref{claim2} there is a positive real \( \varepsilon_0 \) such that \( C_{2\varepsilon_0} \subseteq D \). Pick a closed neighborhood \( D' \in {\EuScript D} \) of \( C \) such that \( D' \subseteq C_{\varepsilon_0} \). Applying Claim~\ref{claim2} again, let \( \varepsilon_1 \) be such that \( C_{\varepsilon_1} \subseteq D' \), and set \( \varepsilon = \min \{ \varepsilon_0, \varepsilon_1 \} \). Finally, pick terms \( t'_i(x\, ;\eta) \) as in Claim~\ref{claim1} applied to \( t_i (x \, ; \eta) \) and the chosen \( \varepsilon \).  Let \( \varphi_0 \in {\EuScript F}_0 \) be the formula \( t'(x\, ;\eta) \in D' \), where \( t'(x\, ; \eta) \) is the the \( n \)-tuple given by the terms \( t'_i(x\, ; \eta) \). Fix a standard structure \( M \) and pick arbitrary \( a \in M^{|x|} \) and \( \alpha \in X^{|\eta|} \). If \( t(a\, ;\alpha) \in C \), then \( t'(a\, ; \alpha) \in C_\varepsilon \subseteq D' \) by choice of \( t' \) and \( \varepsilon \leq \varepsilon_1 \): this proves \( \varphi \to \varphi_0 \). 
Similarly, if \( t'(a\, ;\alpha) \in D' \), then there is \( \beta \in C \) such that \( d_n(t'(a\, ;\alpha),\beta) < \varepsilon_0 \) by \( D' \subseteq C_{\varepsilon_0} \), and thus by choice of \( t' \) and \( \varepsilon\leq \varepsilon_0 \) we also have \( d_n(t(a\, ; \alpha),\beta) \leq d_n(t(a\, ;\alpha),t'(a\, ;\alpha)) + d_n(t'(a\, ;\alpha),\beta) < 2 \varepsilon_0 \), so that \( t(a\, ; \alpha) \in C_{2 \varepsilon_0} \subseteq D \) and we are done.

The inductive steps for \( \wedge \), \( \vee \), \( \exists^{\sf H} \), \( \forall^{\sf H} \), $\existsS$, and $\forallS$ are easy.
\end{proof}

\subsection{Cauchy complete models}\label{Cauchy}

In this section we introduce
abstract model-theoretic versions of some notions that arise naturally in the realm of (pseudo)metric spaces. Our choice of terminology will be justified in Section~\ref{subsec:metric}.

Let $t(x,z)$ be a parameter-free term of sort ${\sf H}^{|x|+|z|}\to {\sf S}$  with \( |z| \) finite, and let $y$ be a tuple of the same length as $x$. We write $x\sim_t y$ as an abbreviation for the formula  $\forall z\,\ \langle t(x,z),t(y,z)\rangle \in \Delta$, where $\Delta$ is the diagonal in $ X\times  X$. By an abuse of notation, we also write $\forall z\, \ t(x,z)=t(y,z)$.
 We let $x\sim_{\sf S} y$  be the type 
 \[
 \left\{x\sim_t y \MID t(x,z)\text{ is a term of sort  }{\sf H}^{|x|+|z|}\to {\sf S} \text{ and } \vert z\vert \in \omega\right\}.
 \]




\begin{proposition}\label{fact_productUniformity} 
 For all natural numbers $n \geq 1$ and all $a=\langle a_i \rangle_{1 \leq i \leq n}$, $b=\langle b_i \rangle_{1 \leq i \leq n}$ in~$\EuScript{U}^n$ 
 \begin{equation}\tag{\( \sharp \)} \label{eq:sharp}
a \sim_{\sf S} b \quad \Leftrightarrow \quad a_i \sim_{\sf S} b_i \text{\ \ \ \ \  for all } 1 \leq i \leq n. 
 \end{equation}
%
%
%
  %
  \end{proposition}
  
  \begin{proof}
  The left-to-right implication is straightforward.  For the converse implication, we proceed by induction on $n \geq 1$. The case \( n = 1 \) is obvious. For the inductive step, we assume that~\eqref{eq:sharp} holds for any two sequences of length up to $n$, and we prove the equivalence for an arbitrary pair of tuples $\langle a_i \rangle_{1 \leq i \leq n+1}$, $\langle b_i \rangle_{1 \leq i \leq n+1}$ satisfying $a_i\sim_{\sf S} b_i$ for all $1 \leq i\leq n+1$. 
The inductive hypothesis gives
\[
\forall y,z\;\; t\big(\langle a_i\rangle_{1 \leq i \leq n},y,z\big) = t\big(\langle b_i\rangle_{1 \leq i \leq n},y,z\big)
\]
for every term $t(x,y,z)$ with $|x|=n$, \( |y| = 1 \), and \( |z| \in \omega \). In particular
\begin{equation} \label{eq:~t1}
\forall z \;\; t\big(\langle a_i\rangle_{1 \leq i \leq n},a_{n+1},z\big) = t\big(\langle b_i\rangle_{1 \leq i \leq n},a_{n+1},z\big).
\end{equation}  
By assumption $a_n\sim_{\sf S} b_n$, so
\[
\forall x,z \;\; t\big(x,a_{n+1},z\big) = t\big(x, b_{n+1}, x\big),
\]
%
and in particular
\begin{equation} \label{eq:~t2}
\forall z \;\; t\big(\langle b_i \rangle_{1 \leq i \leq n}, a_{n+1},z\big) = t\big(\langle b_i \rangle_{1 \leq i \leq n}, b_{n+1},z\big).
\end{equation}
Finally, from~\eqref{eq:~t1} and~\eqref{eq:~t2} we get
\[
\forall z \; t\big(\langle a_i \rangle_{1 \leq i \leq n+1},z\big) = t\big( \langle b_i \rangle_{1 \leq i \leq n+1},z\big). \qedhere
\]
\end{proof}





From now on we assume \( |x| = |y| = n \) and \( |z| \in \omega \).}
 The approximations \( x \sim'_{t} y \) of $x\sim_t y$ have the form
\[
\forall z\, \langle t(x,z),t(y,z)\rangle\in D
\]
for  a closed neighborhood $D$ of  $\Delta$, and will be denoted by \( x \sim_{t,D} y \). We let \( \sim'_{\sf S} \) be the approximation of the type \( \sim_{\sf S} \), namely, the set of formulas
\[
\left\{x\sim_{t,D}y \MID t(x,z)\text{ and } D\text{ are as above}\right\}.
\]

\begin{remark}\label{sim} 
For all tuples of variables \( x ,y \) of the same length 
\[
x\sim_{\sf S} y\quad\Leftrightarrow\quad x\sim'_{\sf S}y.
\]
This equivalence is an instantiation of Proposition~\ref{prop_approx}, and holds in every sufficiently saturated structure.
However, the following topological argument shows that, for the specific types under consideration, the equivalence holds even without any saturation assumptions.
The left-to-right implication is trivial. The converse implication can be proved by contraposition and is a consequence of the normality of $X$. For assume $\langle r,s\rangle\in  X^2\setminus\Delta$. Since  both $\{\langle r,s\rangle\}$ and $\Delta$ are closed sets, there are disjoint open neighborhoods $W_1$ and $W_2$ of $\{\langle r,s\rangle\}$ and $\Delta$, respectively. The complement of $W_1$ is a closed neighborhood of $\Delta$ that does not contain $\langle r,s\rangle$.
\end{remark}

We now show that the formulas \( x \sim_{t,D} y \) yield a subbase of a natural uniformity on \( \EuScript{U}^n \). For this, we first need to close the family of truth sets of such formulas under intersections.
To simplify the notation, if $t(x,z)=\langle t_1(x,z),\dotsc,t_k(x,z)\rangle$ is a tuple of terms and $D=\langle D_1,\dotsc,D_k\rangle$ is a tuple of closed neighborhoods of $\Delta$,  we denote by $x\sim_{t,D}y$  the conjunction  $\bigwedge_{1 \leq i \leq k} x\sim_{t_i,D_i}y$. 
This notation is well behaved with respect to concatenation: if \( t(x,z) \) is the concatenation of the tuple of terms \( t_1(x,z), t_2(x,z) \) and \( D \) is the concatenation of the tuples of closed neighborhoods \( D_1,D_2 \) of $\Delta$ (of the appropriate length), then $x\sim_{t_1,D_1}y\wedge x\sim_{t_2,D_2}y\leftrightarrow x\sim_{t,D}y$.
 As $k$ ranges over the positive natural numbers and $t$ and $D$ range over  the $k$-tuples of terms and the $k$-tuples of  closed neighborhoods of $\Delta$, respectively,   the formulas $x\sim_{t,D}y$ define a family  $\EuScript B_n$ of subsets  of $\EuScript{U}^n\times \EuScript{U}^n$. 
 
 \begin{lemma} \label{lem:uniformity}
\( \EuScript{B}_n \) is a base for a uniformity \( \widehat{U}_n \) on \( \EuScript{U}^n \).
 \end{lemma}
 
 \begin{proof}
 By definition, each element of $\EuScript B_n$ contains the diagonal of $\EuScript{U}^n\times \EuScript{U}^n$.  
  Furthermore, for all $X,Y\in\EuScript B_n$, we have that $X\cap Y$ belongs to  $\EuScript B_n$ by the observation  above on concatenation.
Also, for each $X\in\EuScript B_n$ its inverse $X^{-1}$ belongs to $\EuScript B_n$, since for $E\subseteq  X\times  X$ we have that $E$ is a closed neighborhood of $\Delta$ if and only if so is $E^{-1}$; hence if $X$ is defined by $x\sim_{t,D}y$ then $X^{-1}$ is defined by $x\sim_{t,D^{-1}}y$, where $D^{-1}$ is the tuple obtained by replacing each entry $D_i$ of $D$ with $D_i^{-1}$.

In order to conclude that $\EuScript B_n$ is a base for a uniformity ${\widehat U}_n$ on $\EuScript{U}^n$,  it remains to prove  that,  for each $V\in\EuScript B_n$,  
there exists $W\in\EuScript B_n$ such that $W\circ W\subseteq V$. 
Assume first that  $V$ is defined by $x\sim_{t,D}y$, for a single term \( t \)  and a closed neighborhood $D$ of $\Delta$. Let \( O \) be an open neighborhood of \( \Delta \) such that \( \Delta \subseteq O \subseteq D \). Since $ X$ is compact Hausdorff, it is also divisible, i.e.\ there is a (symmetric) open neighborhood $E$ of $\Delta$ such that  $E \circ E \subseteq O \subseteq D$. By normality of $ X$, there exists a closed neighborhood $F$ of $\Delta$ such that $\Delta\subseteq F\subseteq E$, so that still  $F \circ F \subseteq D$.
It follows that the set $W$ defined by $x\sim_{t,F}y$ is such that $ W\circ W\subseteq V$.

For the general case, let $V$ be the set defined by $x\sim_{t, D}y$, where $t=\langle t_1,\dotsc,t_k\rangle$ is a tuple of terms and $D=\langle D_1,\dotsc,D_k\rangle$ is a tuple of neighborhoods of $\Delta$. Note that $V=\bigcap_{1 \leq i \leq k} V_i$, where each $V_i$ is the set defined by $x\sim_{t_i, D_i} y$. 
By the previous paragraph, for each $1\le i\le k$ we get  a set $W_i$ defined by $x\sim_{t_i, F_i} y$ such that $W_i\circ W_i\subseteq V_i$. Let \( F = \langle F_1, \dotsc, F_k \rangle \). Then the set  $W=W_1\cap\ldots\cap W_k$ defined by $ x\sim_{t, F} y$ is such that $W\circ W\subseteq V$.
%
\end{proof}

The uniformity ${\widehat U}_n$ defines a topology on $\EuScript{U}^n$ which we call the \emph{${\sf S}$-topology}.
Though not needed in what follows, it is worth mentioning that the ${\sf S}$-topology on $\EuScript{U}^{n}$ coincides with the  $n$-th product of the ${\sf S}$-topology on $\EuScript{U}$. This is because, by Proposition~\ref{fact_productUniformity}, the $n$-fold product uniformity of ${\widehat U}_1$ is exactly ${\widehat U}_n$, and it is a known result that the same property is inherited by the associated topologies.   
 
\begin{definition} \label{def:S-invariant}
We say that a type $p(x)\subseteq{\FF}(\EuScript{U})$ is \emph{$\sim_{\sf S}$-invariant}
 if the set defined by \( p(x) \) is invariant with respect to the equivalence relation defined by \( x \sim_{\sf S} y \), i.e.\ if it is true that $x\sim_{\sf S} y\rightarrow(p(x)\leftrightarrow p(y))$. A formula \( \varphi(x) \) is \( \sim_{\sf S} \)-invariant if so is the type \( \{ \varphi(x) \} \).
\end{definition}

By the remark preceding Proposition~\ref{fact_productUniformity},
it is easy to see that all formulas 
not containing any $\LL_{\sf H}$-formula as a subformula 
are $\sim_{\sf S}$-invariant.

The following proposition can be 
regarded as a formulation of the Perturbation Lemma \cite[Proposition~5.15]{HI}. 

\begin{proposition}\label{prop_perturbation} 
For each formula $\varphi(x)\in{\FF}(\EuScript{U})$, the following are equivalent: 
\begin{enumerate-(1)} 
\item \label{prop_perturbation-1} 
$\varphi(x)$ is \emph{$\sim_{\sf S}$}-invariant;
\item \label{prop_perturbation-2}  
for every $\varphi'>\varphi$ there are tuples \( t \) and \( D \) such that
\[
{ {x \sim_{t,D} y} \wedge \varphi(x)}\ \to\ \varphi'(y).
\]
%
\end{enumerate-(1)}  
\end{proposition}

\begin{proof}
%
 
 \ref{prop_perturbation-1} \( \Rightarrow \) \ref{prop_perturbation-2}.
Fix any $\varphi'>\varphi.$ By~\ref{prop_perturbation-1} and Remark~\ref{sim},
we obtain \( {{x \sim'_{\rm S} y} \wedge \varphi(x)} \, \to \, \varphi(y) \). Applying Proposition~\ref{fact_compactness_imp}\ref{fact_compactness_imp-2} gives~\ref{prop_perturbation-2}.

\ref{prop_perturbation-2} \( \Rightarrow \) \ref{prop_perturbation-1}.
Since \( {x \sim_{\sf S} y} \leftrightarrow {y \sim_{\sf S} x} \), it is enough to show that \( x \sim_{\sf S} y \to ( \varphi(x) \to \varphi(y)) \). Fix any \( a,b \in \EuScript{U}^{|x|} \) such that \( a \sim_{\sf S} b \) and \( \varphi(a) \). For each \( \varphi' > \varphi \), fix \( t \) and \(D \) as in~\ref{prop_perturbation-2}. Since \( x \sim_{t,D} y \) is a conjunction of formulas in \( x \sim'_{\sf S} y \), by Remark~\ref{sim} we have \( a \sim_{t,D} b \), and hence \( \varphi'(b) \) by choice of \( t \) and \( D \). By Proposition~\ref{prop_approx}, we conclude that $\varphi(b)$ holds.
\end{proof}

\begin{remark}\label{rem_fin-approx-sat} 
Let \( M \) be a model and let $q(x)\subseteq {\FF}(M)$ be a consistent type.  Unlike the classical setting, if \( M \) is not sufficiently saturated we cannot infer that $q(x)$ is finitely satisfiable in  $M$, but only that $q'(x)$ is finitely satisfiable in $M$. Indeed, we may assume that $q(x)$ is closed under conjunctions. Given $\psi(x) \in q(x)$ we can apply Proposition~\ref{fact_HImorphisms} to $\exists x\, \psi(x)$ and obtain that $\psi'(x)$ is satisfiable in $M$ for each $\psi' > \psi$, as desired.  If instead $M$ is $\omega$-$\FF$-saturated,  than we  also have  that each formula $\psi(x)\in q(x)$ is satisfiable in $M$. In fact, by the previous argument the set  $\{\psi\}'=\{\psi' \MID \psi'>\psi\}$ is finitely satisfiable in $M$, and it contains only finitely many parameters from $M$. Hence $\{\psi\}'$ is satisfiable in $M$ by \( \omega \)-\( \FF \)-saturation and, by Proposition~\ref{prop_approx}, so is  $\psi(x)$. The same proof shows that if a satisfiable type $q(x)$ contains only finitely many parameters from $M$, then it is realized in $M$.
\end{remark}  

\begin{definition} \label{def:cauchy}
Let \( A \subseteq \EuScript{U} \). 
We say that $p(x)\subseteq{\FF}(A)$ is a \emph{Cauchy type} if it is consistent and $p(x)\wedge p(y) \, \rightarrow \, x\sim_{\sf S} y$. 
\end{definition}

\begin{lemma}\label{lem_Cauchyequiv} 
For every model $M$ and every type 
  $p(x)\subseteq{\FF}(M)$, the following are equivalent:
\begin{enumerate-(1)}
\item \label{lem_Cauchyequiv-1} 
$p(x)$ is a $\sim_{\sf S}$-invariant Cauchy type;
 \item \label{lem_Cauchyequiv-2} 
\( p(x) \) is consistent, and $p(x)\leftrightarrow x \sim_{\sf S} a$
 for each $a\in p(\EuScript{U})$. 
 \end{enumerate-(1)}
\end{lemma}

\begin{proof}
 \ref{lem_Cauchyequiv-1} \( \Rightarrow \) \ref{lem_Cauchyequiv-2}. 
 Since \( p(x) \) is a Cauchy type, it is in particular consistent. 
Fix any $a \in p(\EuScript{U})$.
Substituting $a$ for $y$ in the definition of Cauchy type gives $p(x)\rightarrow x \sim_{\sf S} a$. The converse implication follows from \( p(a) \) and the $\sim_{\sf S}$-invariance of $p(x)$. 
  
\ref{lem_Cauchyequiv-2} \( \Rightarrow \) \ref{lem_Cauchyequiv-1}.
Fix any $a \in p(\EuScript{U})$. From  $p(x) \wedge p(y)$, we can derive  $x \sim_{\sf S} a \wedge y\sim_{\sf S} a$ by~\ref{lem_Cauchyequiv-2}. Since the type \( z  \sim_{\sf S} w \) defines an equivalence relation, we conclude that $x\sim_{\sf S} y$. This shows that $p(x)$ is Cauchy. To show that it is also $\sim_{\sf S}$-invariant, assume that $x\sim_{\sf S}y$  and $p(x)$. From~\ref{lem_Cauchyequiv-2}, we get $x\sim_{\sf S} y \wedge x \sim_{\sf S} a$. Hence $y\sim_{\sf S} a$, and so $p(y)$ by~\ref{lem_Cauchyequiv-2} again.  
\end{proof}

\begin{definition} \label{def:Cauchycomplete}
We say that a set $A\subseteq \EuScript{U}$ is \emph{Cauchy complete} if every $\sim_{\sf S}$-invariant Cauchy  type $p(x)\subseteq{\FF}(A)$ is realized by some tuple in $A^{|x|}$.
\end{definition}

\begin{proposition}\label{cauchy-compl-revisited}
  Let $\lambda>\vert{\LL}\vert$ and let $M$ be a  $\lambda$-${\FF}$-saturated model.
  Then $M$ is Cauchy complete.
\end{proposition}

\begin{proof}  
  Let $a\in \EuScript{U}^{\vert x\vert}$ be a realization of the  $\sim_{\sf S}$-invariant Cauchy type $p(x)\subseteq{\FF}(M)$.  We may assume  that $p(x)$ is closed under conjunction.  By Lemma~\ref{lem_Cauchyequiv} we have $p(x)\leftrightarrow x\sim_{\sf S} a$.
  By Proposition~\ref{fact_compactness_imp}\ref{fact_compactness_imp-2}, for every formula $\varphi(x)$ in the type $x \sim_{\sf S} a$ and every $\varphi'>\varphi$ we can pick $\psi(x)\in p(x)$ such that $\psi(x)\rightarrow\varphi'(x)$. 
 Let \( q(x) \subseteq p(x) \) be the type consisting of all such formulas \( \psi(x) \). By Remark~\ref{rem_fin-approx-sat} and saturation of \( M \), any finite conjunction of formulas in \( q(x) \) is satisfiable in \( M \) .
 As there are at most $\vert{\LL}\vert$ formulas in $x \sim_{\sf S} a$, the type \( q(x) \) uses less than \( \lambda \) parameters from \( M \), hence
  by \( \lambda \)-\( \FF \)-saturation of $M$  the type $q(x)$ is realized in \( M \) by some $b \in M^{|x|}$. Then $b \sim_{\sf S}' a$ by definition of \( q(x) \), so $b\sim_{\sf S} a$ by Remark~\ref{sim}. By $\sf S$-invariance of \( p(x) \), we get that $b \in M^{|x|}$ is a realization of $p(x)$.  
\end{proof}

By Corollary~\ref{corol_omega_sat} and \( \lambda \)-\( \FF \)-saturation of \( M \),  we actually get that the element \( b \in M^{|x|} \) from the end of the above proof actually \emph{realizes \( p(x) \) in \( M \)} (and not only in \( \EuScript{U} \)).

Moreover, when  $\LL$ is separable it suffices to require $\lambda>\omega$ in the statement of Proposition~\ref{cauchy-compl-revisited}.

\subsection{Example: bounded metric spaces} \label{subsec:metric}


 Let $ X=[0,1]$ be the real unit interval, equipped with the standard topology. 
Let $\langle M , d \rangle$ be a metric space with \( d \colon M^2 \to X \). For each $n>1$ we also denote by $d$ the supremum metric on $M^n$. 


Let \( \LL \) be a language containing a symbol for \( d \), and possibly for some relations on \( M \) and some uniformly continuous functions from \( M^n \) to \( M \) or from \( M^n \) to \( X \).
Let $\langle M, X\rangle$ be the standard structure that interprets the symbols of \( \LL \) in the natural way.

For the rest of this section, we fix a monster model $M \preccurlyeq^{\FF} \EuScript{U}$ obtained from \( M \) as in the proof of Theorem~\ref{thm:monster}. 
In general,  $d \colon M^2\rightarrow  X$ does not extend to a metric on $\EuScript{U}$. 
For example, if $d$ takes arbitrarily small positive values, by compactness we can find in \( \EuScript{U} \) pairs of distinct elements at distance $0$ from one another. However, it is true that the interpretation of $d$ in \( \EuScript{U} \) is a pseudometric, that we still denote by $d$. 
Therefore the notions of uniformly continuous function, of convergent sequence, and alike  make sense in all finite powers of $\EuScript{U}$.
 To simplify the notation, for every \( \varepsilon \in [0,1] \) we use \( d(x,y) \leq \varepsilon \) as an abbreviation for the \( \FF \)-atomic formula \( d(x,y) \in [0,\varepsilon] \). This is naturally adapted to product spaces. For all \( n \geq 1 \), the product $\EuScript{U}^n$ is naturally equipped with the supremum pseudometric, still denoted by \( d \), which obviously extends the corresponding supremum metric on \( M^n \). For tuples  of variables \( x = \langle x_i \rangle_{1 \leq i \leq m} \) and \( \langle y \rangle_{1 \leq i \leq m} \) of the same finite length, we use \( d(x,y) \leq \varepsilon \) to abbreviate the conjunction \( \bigwedge_{1 \leq i \leq m} d(x_i,y_i) \leq \varepsilon \).

The uniform continuity of a given function $f \colon M^n\to M$ with \( f \in \LL \) can be formulated using a set of \( \FF \)-formulas. Indeed, by uniform continuity for all $0<\varepsilon\le 1$ there exists $0<\delta\le 1$ such that the following \( \FF \)-formula is satisfied in \( M \):
\[
\forall x \forall y \, \Big(d(x,y)\in [\delta, 1] \vee d(f(x),f(y))\in [0,\varepsilon]\Big).
\]
The \( \FF \)-type of all such sentences (one for each \( 0 <  \varepsilon \leq 1 \), paired with any suitable \( 0 < \delta \leq 1 \)) is preserved from \( M \) to any monster model \( \EuScript{U} \) with \( M \preccurlyeq^{\FF} \EuScript{U} \), hence the interpretation of \( f \) in \( \EuScript{U} \) remains uniformly continuous.
Similar considerations apply to uniformly continuous functions from $M^n$ to $ X$ belonging to \( \LL \), as well as to the functions defined in \( M \) by \( \LL \)-terms (which are necessarily uniformly continuous because by choice of \( \LL \) they are compositions of finitely many uniformly continuous functions).

\begin{lemma}\label{rmk_uncon}  Let $x$ and $y$ be tuples of variables of sort $\sf H$ and of equal length \( m \geq 1 \). Then 
\[
x\sim_{\sf S}y\leftrightarrow d(x,y)=0.
\]
\end{lemma}
\begin{proof}
Consider arbitrary \( a,b \in \EuScript{U}^m \). 
If $a\sim_{\sf S}b$, then $\forall z \,\ d(a,z)=d(b,z)$. Hence $d(a,a)=d(b,a)$, from which $d(a,b)=0$ follows.

For the converse implication, assume that $d(a,b)=0$. Let $t(x,z)$ be any term of sort ${\sf H}^{m+k}\rightarrow {\sf S}$ for some $k\in\omega$. 
 As observed before this lemma, the term \( t(x,z) \) defines a uniformly continuous function in \( \EuScript{U} \), hence for all $c \in \EuScript{U}^k$ we have that $\vert  t(a,c),t(b,c)\vert\le 1/n$ holds in $\EuScript{U}$ for all $n > 0$. 
 Therefore \( \forall z \,\ t(a,z)=t(b,z) \), and since \( t \) was arbitrary we get \( a \sim_{\sf S} b \).
\end{proof}

\begin{proposition} \label{prop:completemetricspaces}
Let $N \preccurlyeq^{\FF} \EuScript{U}$ be a model and let \( m \geq 1 \). For every $a\in \EuScript{U}^m$, the following are equivalent:
  \begin{enumerate-(1)}
    \item \label{prop:completemetricspaces-1}
    $p(x)={\FF}\mbox{-tp}(a/N)$ is a Cauchy type;
    \item \label{prop:completemetricspaces-2}
    \( a \in \EuScript{U}^m \) is a realization of an \( \sim_{\sf S} \)-invariant Cauchy type \( q(x) \subseteq \FF(N) \);
    \item \label{prop:completemetricspaces-3} 
    there is a sequence $\langle a_k\rangle_{k \in \omega} $ in \( N^m \)  that converges to $a$ with respect to \( d \).
    \end{enumerate-(1)}
  \end{proposition} 

\begin{proof}
\ref{prop:completemetricspaces-1} \( \Rightarrow \) \ref{prop:completemetricspaces-3}. 
  As $p(x)$ is a Cauchy type, we have $p(x)\rightarrow x\sim_{\sf S} a$ (see the first part of the  proof of \ref{lem_Cauchyequiv-1} \( \Rightarrow \) \ref{lem_Cauchyequiv-2} in Lemma~\ref{lem_Cauchyequiv}, which does not require \( \sim_{\sf S} \)-invariance).
 By Lemma~\ref{rmk_uncon} and Proposition~\ref{prop_approx},  we have $p'(x)\rightarrow d(a,x)\le 2^{-k}$, for all $k \in \omega$.
 Since for \( k \geq 1 \) the interval \( [0,2^{-k+1}] \) is a closed neighborhood of \( [0,2^{-k}] \) in the space \( X = [0,1] \),
by  Proposition~\ref{fact_compactness_imp}\ref{fact_compactness_imp-2} applied to \( p'(x) \) and \( d(a,x) \leq 2^{-k} \) we get that for each $k\geq 1$ there is a formula $\varphi_k(x)\in p(x)$ and an approximation \( \varphi'_k(x) \in p'(x) \) of it such that $\varphi'_k(x)\rightarrow d(a,x)\le 2^{-k+1}$. (Here we use that \( p(x) \) and \( p'(x) \) are closed under finite conjunctions.) Since \( \varphi_k(x) \) is satisfiable in \( \EuScript{U} \), as witnessed by \( a \), by Remark~\ref{rem_fin-approx-sat} for each $k\in\omega$ there  is  $a_k\in M$ such that $M \models \varphi'_k(a_k)$, and hence also \( \EuScript{U} \models \varphi'_k(a_k) \).
  As $d(a,a_k)\le 2^{-k+1}$ for all $k \geq 1$, the sequence $\langle a_k \rangle_{k \in\omega}$ converges to $a$ with respect to the pseudometric \( d \).

\ref{prop:completemetricspaces-3} \( \Rightarrow \) \ref{prop:completemetricspaces-2}.
For each $k\in\omega$, let $\varepsilon_k=d(a_k,a)$.
  Then $\langle \varepsilon_k\rangle_{k\in \omega}$ is a sequence of non-negative reals  converging to $0$.    
  Let \( q(x) = \big\{ d(a_k,x) \leq \varepsilon_k \MID k \in \omega \big\} \).
  All the \( \FF \)-atomic formulas $d(a_k,x)\leq\varepsilon_k$ are \( \sim_{\sf S} \)-invariant because they do not contain \( \LL_{\sf H} \)-subformulas, hence \( q(x) \) is \( \sim_{\sf S} \)-invariant too, and \( a \) realizes \( q(x) \). It remains to prove that \( q(x) \) is a Cauchy type. 
It is obviously consistent, as witnessed by \( a \).
Moreover, by the triangle inequality every realization \( b \in \EuScript{U} \) of \( d(a_k,x) \leq \varepsilon_k \) is such that \( d(b,a) \leq 2 \varepsilon_k \). Thus every realization of $q(x)$ is at distance $0$ from $a$,
hence \( q(x) \wedge q(y) \to d(x,y) = 0 \) by the triangle inequality.
By Lemma~\ref{rmk_uncon}, this means that \( q(x) \wedge q(y) \to x \sim_{\sf S} y \) and we are done.  

\ref{prop:completemetricspaces-2} \( \Rightarrow \) \ref{prop:completemetricspaces-1}.
Clear, since \( q(x) \subseteq p(x) \).
\end{proof}

\begin{proposition} \label{prop:cauchycompletenessformetric}
For every model \(N \preccurlyeq^{\FF} \EuScript{U} \), the following are equivalent:
\begin{enumerate-(1)}
\item \label{prop:cauchycompletenessformetric-1}
\( N \) is Cauchy complete;
\item \label{prop:cauchycompletenessformetric-2}
\( N \) is \( d \)-Cauchy complete, that is, every \( d \)-Cauchy sequence in \( N \) converges to a point in \( N \).
\end{enumerate-(1)}
\end{proposition}

\begin{proof}
\ref{prop:cauchycompletenessformetric-1} \( \Rightarrow \) \ref{prop:cauchycompletenessformetric-2}.
Let \( \langle a_k \rangle_{k \in \omega} \) be a \( d \)-Cauchy sequence in \( N \). By working with a subsequence if necessary, we may assume that \( d(a_i,a_j) \leq 2^{-k} \) for every \( k \in \omega \) and every \( i,j \geq k \). 
The \( \FF \)-type \( p(x) =  \{ d(a_k,x) \leq 2^{-k} \mid k \in \omega \} \) is finitely realized in \( \EuScript{U} \): if \( k \) is largest such that \( d(a_k,x) \leq 2^{-k} \) occurs in a given finite subset of \( p(x) \), then \( a_k \) is a realization of such a set. By saturation,
\( p(x) \subseteq \FF(N) \) is realized by some \( a \in \EuScript{U} \), and it is thus consistent. Moreover, it is \( \sim_{\sf S} \)-invariant because none of its elements contain an \( \LL_{\sf H} \)-subformula, and \( p(x) \wedge p(y) \to x \sim_{\sf S} y \) by the proof of  \ref{prop:completemetricspaces-3} \( \Rightarrow \) \ref{prop:completemetricspaces-2} in Proposition~\ref{prop:completemetricspaces} (where now we set \( \varepsilon_k = 2^{-k} \)). Thus \( p(x) \) is realized by some \( b \in N \) by~\ref{prop:cauchycompletenessformetric-1}. By definition of \( p(x) \), the sequence \( \langle a_k \rangle_{k \in \omega} \) converges to \( b \) and we are done.

\ref{prop:cauchycompletenessformetric-2} \( \Rightarrow \) \ref{prop:cauchycompletenessformetric-1}.
Notice that \( N \) is \( d \)-Cauchy complete if and only if so is each finite product \( N^m \) for \( m \geq 1 \).
Let \( q(x) \subseteq \FF(N) \), with \( |x| = m \), be an \( \sim_{\sf S} \)-invariant Cauchy type, and let
\( a \in \EuScript{U}^{m} \) be a realization of it. 
By Proposition~\ref{prop:completemetricspaces}, there is a sequence \( \langle a_k \rangle_{k \in \omega} \) in \( N^{m} \) which converges to \( a \). We may assume that \( \langle a_k \rangle_{k \in \omega} \) is \( d \)-Cauchy; thus it also converges to some \( b \in N^{m} \) by~\ref{prop:cauchycompletenessformetric-2}. Since both \( a \) and \( b \) are limits of the same sequence, we have \( d(a,b) = 0 \), and hence \( a \sim_{\sf S} b \) by Lemma~\ref{rmk_uncon}. Since \( q(x) \) is \( \sim_{\sf S} \)-invariant, \( b \in N^{m} \) realizes \( q(x) \).
\end{proof}

In particular, by Proposition~\ref{cauchy-compl-revisited} we get that \( \langle \EuScript{U},d \rangle \) is complete as a pseudometric space.

\newcommand\biburl[1]{\url{#1}}
\BibSpec{arXiv}{%
  +{}{\PrintAuthors}{author}
  +{,}{ \textit}{title}
  +{}{ \parenthesize}{date}
  +{,}{ arXiv:}{eprint}
}

\end{document}